\documentclass[11pt,twoside]{amsart}
\usepackage{amssymb}
\let\llrcorner\lrcorner

\usepackage[latin1]{inputenc}
\usepackage[T1]{fontenc}
\usepackage{mathtools}
\usepackage{graphicx}
\usepackage{tikz}
\usepackage{mathabx}
\usetikzlibrary{chains}
\usepackage[latin1]{inputenc}
\usepackage{amsmath}
\usepackage{amsthm}

\usepackage[all]{xy}
\usepackage{hyperref}
\newtheorem{thm}{Theorem}[section]

\newtheorem{prop}[thm]{Proposition}

\newtheorem{lem}[thm]{Lemma}
\newtheorem{cor}[thm]{Corollary}

\numberwithin{equation}{section}

\theoremstyle{definition}
\newtheorem{definition}[thm]{Definition}
\newtheorem{remark}[thm]{Remark}
\newtheorem{conj}[thm]{Conjecture}
\newtheorem{ex}[thm]{Example}

\newcommand{\Db}{{\rm D}^{\rm b}}

\newcommand{\Aut}{{\rm Aut}}

\newcommand{\rk}{{\rm rk}}

\newcommand{\Hom}{{\rm Hom}}

\newcommand{\Ext}{{\rm Ext}}

\newcommand{\K}{{\rm K}}

\newcommand{\cal}{\mathcal}
\newcommand{\ka}{{\cal A}}

\newcommand{\kc}{{\cal C}}

\newcommand{\ke}{{\cal E}}

\newcommand{\kl}{{\cal L}}

\newcommand{\ko}{{\cal O}}

\newcommand{\kt}{{\cal T}}
\newcommand{\ks}{{\cal S}}
\newcommand{\kx}{{\cal X}}

\newcommand{\ZZ}{\mathbb{Z}}
\newcommand{\QQ}{\mathbb{Q}}
\newcommand{\RR}{\mathbb{R}}
\newcommand{\CC}{\mathbb{C}}

\newcommand{\HoH}{H\!H}
\newcommand{\PP}{\mathbb{P}}

\newcommand{\OO}{{\rm O}}

\renewcommand{\to}{\xymatrix@1@=15pt{\ar[r]&}}
\newcommand{\lto}{\xymatrix@1@=15pt{&\ar[l]}}
\renewcommand{\rightarrow}{\xymatrix@1@=15pt{\ar[r]&}}
\renewcommand{\mapsto}{\xymatrix@1@=15pt{\ar@{|->}[r]&}}
\newcommand{\mapslto}{\xymatrix@1@=15pt{&\ar@{|->}[l]&}}
\renewcommand{\twoheadrightarrow}{\xymatrix@1@=18pt{\ar@{->>}[r]&}}
\renewcommand{\hookrightarrow}{\xymatrix@1@=15pt{\ar@{^(->}[r]&}}
\newcommand{\hook}{\xymatrix@1@=15pt{\ar@{^(->}[r]&}}
\newcommand{\congpf}{\xymatrix@1@=15pt{\ar[r]^-\sim&}}
\renewcommand{\cong}{\simeq}

\makeatletter
\def\blfootnote{\xdef\@thefnmark{}\@footnotetext}
\makeatother

\begin{document}

\title{The K3 category of a cubic fourfold}

\author[D.\ Huybrechts]{Daniel Huybrechts}

\address{Mathematisches Institut,
Universit{\"a}t Bonn, Beringstr.\ 1, 53115 Bonn, Germany}
\email{huybrech@math.uni-bonn.de}

\begin{abstract} \noindent
Smooth cubic hypersurfaces $X\subset \PP^5$ (over $\CC$) are  linked to K3 surfaces  via their Hodge structures, due to work of Hassett, and via a subcategory $\ka_X\subset\Db(X)$, due to work of Kuznetsov.
The relation between these two viewpoints has recently been elucidated by  Addington and Thomas.
 
In this paper, both aspects are studied further and extended to  twisted K3 surfaces, which in particular
allows us to determine the group of autoequivalences of $\ka_X$ for the  general cubic fourfold.
Furthermore, we prove finiteness results for
cubics with equivalent K3 categories and study periods of cubics in terms of generalized K3 surfaces.
 \vspace{-2mm}
\end{abstract}
%\dedicatory{}

\maketitle
\blfootnote{This work was supported by the SFB/TR 45 `Periods,
Moduli Spaces and Arithmetic of Algebraic Varieties' of the DFG
(German Research Foundation).}

\marginpar{}

%%%%%%%%%%%%%%%%%%%%%%%%%%%%%
\section{Introduction}
As shown by Kuznetsov  \cite{Kuz1,Kuz5}, the bounded derived category $\Db(X)$ of coherent sheaves on a smooth cubic hypersurface
$X\subset\PP^5$ contains, as the semi-orthogonal complement of  three line bundles,
a full triangulated subcategory $$\ka_X\coloneqq\langle\ko,\ko(1),\ko(2)\rangle^\perp\subset\Db(X)$$ that behaves in many respects like the bounded
derived category $\Db(S)$ of coherent sheaves on a K3 surface $S$. In fact, for certain special cubics $\ka_X$ is
equivalent to $\Db(S)$ or, more generally, to the derived category $\Db(S,\alpha)$ of $\alpha$-twisted
sheaves on a K3 surface $S$. Kuznetsov also conjectured that $\ka_X$ is of the form $\Db(S)$ if and only
if $X$ is rational. Neither of the two implications has been verified until now, although Addington and Thomas recently have shown in
\cite{AT} that the conjecture is (generically) equivalent to a conjecture attributed to Hassett \cite{HassettComp} describing rational cubics in terms of their periods.

\subsection{} This paper is not concerned with the rationality of cubic fourfolds, but with basic results on $\ka_X$. Ideally, one would like
to have a theory for $\ka_X$ that parallels the theory for $\Db(S)$ and $\Db(S,\alpha)$. In particular, one would like to have analogues
of the following results and conjectures:
\medskip

-- \emph{For a given twisted K3 surface $(S,\alpha)$ there exist only finitely many isomorphism
classes of twisted K3 surfaces $(S',\alpha')$ with $\Db(S,\alpha)\cong\Db(S',\alpha')$.}
\smallskip

-- \emph{Two twisted K3 surfaces $(S,\alpha)$, $(S',\alpha')$ are derived equivalent,
i.e.\ there exists a $\CC$-linear exact equivalence $\Db(S,\alpha)\cong\Db(S',\alpha')$, if and only if there
exists an orientation preserving Hodge isometry $\widetilde H(S,\alpha,\ZZ)\cong\widetilde H(S',\alpha',\ZZ)$.}
\smallskip

-- \emph{The group of linear exact autoequivalences of $\Db(S,\alpha)$ admits a natural representation $\rho\colon\Aut(\Db(S,\alpha))\to\Aut(\widetilde H(S,\alpha,\ZZ))$,
which is surjective up to index two. Moreover, up to finite index  $\Aut(\Db(S,\alpha))$  is
conjecturally described as a fundamental group of a certain Deligne--Mumford stack.}

\medskip

Most of the theory for untwisted K3 surfaces is due to Mukai \cite{MukaiVB} and Orlov \cite{OrlovK3}, whereas the basic theory of twisted K3 surfaces was
developed in \cite{HuSt,HuSt2}. See also \cite{HuyFM,HuSeat} for surveys and further references.  Originally, the generalization to twisted K3 surfaces was motivated by
the existence of non-fine moduli spaces \cite{Caltw}. However, more recently it has become clear that allowing twists has quite unexpected applications,
e.g.\ to the Tate conjecture \cite{Char,MLS}. Crucial for the purpose of this article is the observation proved together with Macr\`i and Stellari
\cite{HMS2}
that ${\rm Ker}(\rho)=\ZZ[2]$ for many
twisted K3 surfaces $(S,\alpha)$. Note that for untwisted projective K3 surfaces the kernel is always highly non-trivial and,
in particular, not finitely generated. The conjectural
description of $\Aut(\Db(S))$ has in this case only been achieved for K3 surfaces of Picard rank one \cite{BB}.

\subsection{} As a direct attack on $\ka_X$ is difficult, we follow Addington and Thomas \cite{AT} and
reduce the study of $\ka_X$ via deformation to the case of  (twisted) K3 surfaces. Central to our discussion
is the  Hodge structure $\widetilde H(\ka_X,\ZZ)$ associated with $\ka_X$  introduced in \cite{AT} as the
analogue of the Mukai--Hodge structure $\widetilde H(S,\ZZ)$ of weight two on the full cohomology $H^*(S,\ZZ)$ of a K3 surface $S$ or of the twisted
version $\widetilde H(S,\alpha,\ZZ)$ introduced in \cite{HuSt}. For example, 
any FM-equivalence $\ka_X\cong\ka_{X'}$ induces a Hodge isometry $\widetilde H(\ka_X,\ZZ)\cong\widetilde H(\ka_{X'},\ZZ)$, cf.\ Proposition
\ref{prop:equivHodgeA}. This suffices to prove:

\begin{thm}\label{thm:finiteFM}
For any given smooth cubic $X\subset \PP^5$ there are only finitely many cubics $X'\subset \PP^5$ up to isomorphism
for which there exists a FM-equivalence $\ka_X\cong\ka_{X'}$. See Corollary \ref{cor:finiteFM}.
\end{thm}

Recall that due to a result of Bondal and Orlov a smooth cubic $X\subset\PP^5$ itself does not admit any non-isomorphic Fourier--Mukai partners.
This is no longer true if $\Db(X)$ is replaced by its K3 category $\ka_X$. In particular, there exist FM-equivalences $\ka_X\cong\ka_{X'}$ that do not extend to
equivalences $\Db(X)\cong\Db(X')$. However, we will also see that  general cubics $X$ and $X'$, i.e.\ those for which
$\rk\,H^{2,2}(X,\ZZ)=1$,  admit a FM-equivalence
$\ka_X\cong\ka_{X'}$ if and only if $X\cong X'$, see Theorem \ref{thm:HodgeA} or Corollary \ref{cor:verygeneralnoFM}.
%Furthermore, the number of isomorphism classes of cubics $X'$ for which there exists
%a FM-equivalence $\ka_X\cong\ka_{X'}$ can be determined for the generic cubic $X$ with $\rk\, H^{2,2}(X,\ZZ)=2$, see Corollary
%\ref{cor:numberFM}.

%For the same reason that \cite{AT} proves the equivalence of the two rationality conjectures due to Hassett resp.\ Kuznetsov only
%generically, most  results in this paper are only proved for generic cubics or even only general ones.
%
%
%\begin{thm}\label{thm:HodgeA}
%For general smooth cubics $X,X'\subset\PP^5$
%there exists a Hodge isometry $\widetilde H(\ka_X,\ZZ)\cong\widetilde H(\ka_{X'},\ZZ)$
%if and only if there exists a FM-equivalence $\ka_X\cong\ka_{X'}$.
%\end{thm}

\smallskip

The following can be seen as an easy analogue of the result of Bayer and Bridgeland \cite{BB} describing
$\Aut(\Db(S))$ for  general K3 surfaces $S$ (namely those with $\rho(S)=1$) or rather of \cite{HMS} describing this
group for general non-projective K3 surfaces or twisted projective K3 surfaces $(S,\alpha)$ without $(-2)$-classes (see Section \ref{sec:ProofAut}).
 
\begin{thm}\label{thm:noFMvg}
{\rm i)} For the very general\footnote{A property holds for
the \emph{very general} cubic if it holds for cubics in the complement of countably many
proper closed subsets of the space of cubics under consideration. It holds
for the \emph{generic} cubic if it holds for a Zariski open, dense subset.} smooth cubic $X\subset\PP^5$ the group $\Aut_s(\ka_X)$
of symplectic FM-autoequivalences
is infinite cyclic with
$$\Aut_s(\ka_X)/\ZZ\cdot[2]\cong\ZZ/3\ZZ.$$
Alternatively, the group of all FM-autoequivalences $\Aut(\ka_X)$ is infinite cyclic containing
$\ZZ\cdot[1]$ as a subgroup of index three.

{\rm ii)}
Moreover, the induced action on $\widetilde H(\ka_X,\ZZ)$ of any FM-autoequivalence $\Phi\colon\ka_X\congpf\ka_X$ of a non-special cubic
preserves the natural orientation.
%ii) For a dense Zariski open set of cubics $X\subset\PP^5$ the
%image of the natural map $$\rho\colon\Aut(\ka_X)\to \Aut(\widetilde H(\ka_X,\ZZ))$$
%is $\Aut^+(\widetilde H(\ka_X,\ZZ))$, the index two subgroup of orientation preserving Hodge isometries.
\end{thm}

In fact,  for every smooth cubic $\Aut_s(\ka_X)$ contains an infinite cyclic group with $\ZZ\cdot[2]$ as a subgroup of index three,  see Corollary \ref{cor:Z3Z}.
The theory of twisted K3 surfaces is crucial for the theorem, as eventually the problem is reduced to \cite{HMS2} which deals with general twisted K3 surfaces. 

The group $\Aut_s(\ka_X)$ of an arbitrary cubic is described by an analogue of Brigdeland's conjecture, 
see Conjecture \ref{conj:Brid}.
%The second part is the cubic version of the main result of \cite{HMS}.

\subsection{} In \cite{HassettComp} Hassett showed that in the moduli space $\kc$
of smooth cubics, the set of those cubics $X$ for which there exists a primitive positive plane   $K_d\subset H^{2,2}(X,\ZZ)$
of discriminant $d$ containing the class ${\rm c}_1(\ko(1))^2$ is
an irreducible divisor $\kc_d\subset\kc$. Moreover, $\kc_d$ is not empty if and only if\vskip0.2cm

($\ast$) $d\equiv 0,2\, (6)$ and $d>6$.
\vskip0.2cm

Cubics parametrized by the divisors $\kc_d$ are called special.
Hassett also introduced the numerical condition\vskip0.2cm

($\ast$$\ast$) $d$ is even and $d/2$ is not divisible by $9$ or any prime
$p\equiv 2\, (3)$.\footnote{This condition was  originally stated as:
$d\equiv 0,2\, (6)$ and $d$ not divisible by $4,9$ or any prime $2\ne p\equiv 2\, (3)$.
The reformulation has been suggested by the referee.}
\vskip0.2cm \noindent

\noindent and proved that ($\ast$$\ast$) is equivalent to the orthogonal complement of the corresponding lattice $K_d$
in $H^4(X,\ZZ)$ being (up to sign) Hodge isometric to the 
primitive Hodge structure $H^2(S,\ZZ)_{\rm prim}$ of a polarized K3 surface. In \cite{AT} the condition was shown to be equivalent to the existence of
a Hodge isometry $\widetilde H(\ka_X,\ZZ)\cong \widetilde H(S,\ZZ)$ for some K3 surface $S$. We prove the following twisted version of it
(cf.\ Proposition \ref{prop:DtwDd}):

\begin{thm}\label{thm:twHas}
For a smooth cubic $X\subset \PP^5$ the Hodge structure $\widetilde H(\ka_X,\ZZ)$ is Hodge isometric
to the Hodge structure $\widetilde H(S,\alpha,\ZZ)$ of a twisted K3 surface $(S,\alpha)$ if and only if 
$X\in \kc_d$ with
%there exists a positive plane
%$K_d\subset H^{2,2}(X,\ZZ)$ of discriminant $d$ containing the class ${\rm c}_1(\ko(1))^2$ such that
\vskip0.2cm
\emph{($\ast$$\ast'$)} $d$ is even and  in 
the prime factorization $d/2=\prod p_i^{n_i}$ one has $n_i\equiv0\, (2)$ for all primes $p_i\equiv2\, (3)$.
\end{thm}
Obviously, if $d$ satisfies ($\ast$$\ast$), then $k^2d$ satisfies ($\ast$$\ast'$) for all integers $k$.
Conversely, any $d$ satisfying  ($\ast$$\ast'$) can be written (in general non-uniquely)
as $k^2d_0$ with $d_0$ satisfying ($\ast$$\ast$).

Also note that for $X\in{\mathcal C}_d$ with $d$ satisfying
($\ast$$\ast'$) the transcendental part $T(X)\subset H^{2,2}(X,\ZZ)$ is Hodge isometric (up to sign)
to $T(S,\alpha)$ of a twisted K3 surface $(S,\alpha)$ (cf.\ \cite{HuSt}):
\begin{equation}
T(X)(-1)\cong T(S,\alpha)\cong {\rm Ker}(\alpha\colon T(S)\to \QQ/\ZZ).
\end{equation}

As the main result of \cite{AT}, Addington and Thomas proved that at least generically
($\ast$$\ast$) is equivalent to $\ka_X\cong\Db(S)$ for some K3 surface $S$. The following
twisted version of it will be proved in Section \ref{sec:thm:genastastast}.

%Combined with Theorem \ref{thm:HodgeA} this eventually leads to

\begin{thm}\label{thm:genastastast}
{\rm i)} If $\ka_X\cong\Db(S,\alpha)$ for some twisted K3 surface $(S,\alpha)$, then $X\in\kc_d$ with $d$ satisfying \emph{($\ast$$\ast'$)}.

{\rm ii)} Conversely, if $d$ satisfies  \emph{($\ast$$\ast'$)}, then there exists a Zariski open dense set $\emptyset\ne U\subset\kc_d$ such that
for all $X\in\kc_d$ there exists a twisted K3 surface $(S,\alpha)$  and an equivalence $\ka_X\cong\Db(S,\alpha)$.
\end{thm}

Non-special cubics are determined by their associated K3 category $\ka_X$ and for general special cubics
$\ka_X$ is determined by its Hodge structure (see Corollary \ref{cor:verygeneralnoFM} and
 Section \ref{sec:proofthm:HodgeA}):

\begin{thm}\label{thm:HodgeA}
 Let $X$ and $X'$ be two smooth cubics.
 
{\rm i)} Assume $X$ is not special, i.e.\ not contained in any $\kc_d\subset\kc$.
Then there exists a FM-equivalence $\ka_X\cong\ka_{X'}$ if and only
if $X\cong X'$. 

{\rm ii)} For $d$ satisfying \emph{($\ast$$\ast'$)} and a  Zariski dense open set of 
%\footnote{More precisely, for $X\in U\setminus\bigcup_{d'\ne d}\kc_{d'}$, where 
%$\emptyset\ne U\subset\kc_d$ is Zariski open. For $d$ satisfying ($\ast$$\ast$) or ($\ast$$\ast$$\ast$) the assertion holds
%for generic $X\in\kc_d$, see below.} 
cubics $X\in \kc_d$,  there exists a FM-equivalence $\ka_X\cong\ka_{X'}$ if and only if there exists a Hodge isometry 
$\widetilde H(\ka_X,\ZZ)\cong\widetilde H(\ka_{X'},\ZZ)$.

{\rm iii)} For arbitrary $d$ and very general $X\in\kc_d$   there exists
a FM-equivalence $\ka_X\cong\ka_{X'}$ if and only if there exists a Hodge isometry 
$\widetilde H(\ka_X,\ZZ)\cong\widetilde H(\ka_{X'},\ZZ)$. 
\end{thm}

We will also see that  arguments of Addington \cite{Add} can be adapted to show that ($\ast$$\ast'$) is in fact equivalent
to the Fano variety of lines on $X$ being birational to a moduli space of twisted sheaves on some K3 surface, see Proposition \ref{prop:analAdd}.

\subsection{} There are a few fundamental issues concerning $\ka_X$ that we do not know how to address and that prevent us from developing the theory in full. Firstly, this paper only deals with FM-equivalences $\ka_X\cong\ka_{X'}$, i.e.\ those for which the composition
$\Db(X)\to\ka_X\congpf\ka_{X'}\,\hookrightarrow\Db(X')$ is a Fourier--Mukai transform. One would expect this to be the case for
all equivalences, but the classical result of Orlov \cite{OrlovK3} and its generalization by Canonaco and Stellari \cite{CanSte} do not apply to this situation.
Secondly, it is not known whether $\ka_X$ always admits bounded t-structures or stability conditions. This is problematic when one
wants to study FM-partners of $\ka_X$ as moduli spaces of (stable) objects in $\ka_X$. As in \cite{AT}, the lack of stability is also the crucial stumbling block to use deformation theory to prove statements as in Theorem \ref{thm:genastastast} for all cubics and not only for generic or very general ones. 

\subsection{} The plan of the paper is as follows. Section \ref{sec:Lattice} deals with all issues related to the lattice theory and
the abstract Hodge theory. In particular, natural (countable unions of) codimension one subsets
$D_{\rm K3}\subset D_{\rm K3'}$ of the period domain $D\subset \PP(A_2^\perp\otimes\CC)$ are studied at great length. They parametrize
periods that induce Hodge structures that are Hodge isometric  to $\widetilde H(S,\ZZ)$ and $\widetilde H(S,\alpha,\ZZ)$,
respectively, and which are
described in terms of the numerical conditions ($\ast$$\ast$) and ($\ast$$\ast'$). In particular, Theorem \ref{thm:twHas} is proved.
We also provide a geometric description
of all periods $x\in D$ in terms of generalized K3 surfaces, see Proposition \ref{prop:genK3allD}.

In Section \ref{sec:AX} we extend results in \cite{AT} from equivalences $\ka_X\cong\Db(S)$ to the twisted case and prove the finiteness of FM-partners for $\ka_X$, see Theorem \ref{thm:finiteFM}.
Moreover, we produce an action of the universal cover of ${\rm SO}(A_2)$ on $\ka_X$ for all cubics
(Remark \ref{rem:univcoveract}) and formulate an analogue of Bridgeland's conjecture (cf.\ Conjecture \ref{conj:Brid}).
%and bound the number of FM-partners of $\ka_X$ for the generic special cubic (Corollary \ref{cor:numberFM}).

The short Section \ref{sec:Fano} shows that ($\ast$$\ast'$) is equivalent to $F(X)$ being birational to a moduli space of
stable twisted sheaves on a K3 surface. In Section \ref{sec:defo} we adapt the deformation theory of \cite{AT} to the twisted case.
Finally, in Section \ref{sec:Proofs} we conclude the proofs of Theorems \ref{thm:noFMvg}, \ref{thm:genastastast}, and \ref{thm:HodgeA}.

\smallskip

%{\bf Acknowledgements:} 
\subsection{Acknowledgements} I would like to thank Nick Addington and Sasha Kuznetsov for very helpful discussions
during the preparation of the paper.  I am also grateful to Ben Bakker, Daniel Halpern-Leistner,
J\o rgen Rennemo, Paolo Stellari, Andrey Soldatenkov,
and Richard Thomas for comments and suggestions. I enjoyed several discussions with Alex Perry, in particular
on the possibility of proving a result like Theorem \ref{thm:finiteFM}, for which he has also found a proof.
Thanks to Emanuel Reinecke and Pawel Sosna  for a long list of
comments on the first version and to the referee for a very careful reading and innumerable suggestions.

%%%%%%%%%%%%%%%%%%%%%%%%%%%%%%%%%%%
\section{Lattice theory and period domains}\label{sec:Lattice}

We start by discussing the relevant lattice theory. To make the reading self-contained,
we will also recall results due to Hassett and to Addington and Thomas on the way.

There are two kinds of lattices, those related to K3
surfaces, $\Lambda$, $\widetilde\Lambda$, etc., and those attached to cubic fourfolds,
${\rm I}_{21,2}$, $K_d$, etc.. 
The two types are linked by a lattice $A_2^\perp$ of signature $(2,20)$ and two embeddings
$$\xymatrix{{\rm I}_{2,21}&A_2^\perp\ar@{^(->}[r] \ar@{_(->}[l]&\widetilde\Lambda.}$$
The induced maps between the associated period domains allows one to relate periods of cubic fourfolds to
periods of (generalized) K3 surfaces.

\subsection{}\label{sec:lattpreps} By $U$ we shall denote the hyperbolic plane, i.e.\  $\ZZ^2$ with the intersection
form  \scalebox{0.7}{$\left(\begin{array}{cc} 0& 1 \\1 & 0\end{array}\right)$}. The K3 lattice $\Lambda$ and 
the extended K3 lattice $\widetilde\Lambda$ are by definition the unique even, unimodular lattice of signature $(3,19)$ and $(4,20)$, respectively. So, 
$$\Lambda\cong E_8(-1)^{\oplus 2}\oplus U^{\oplus 3}\text{ and  }
\widetilde\Lambda\cong \Lambda\oplus U.$$ 

Next, $A_2$  denotes the standard root lattice of rank two, i.e.\ there exists
a basis $\lambda_1,\lambda_2$ with respect to which the intersection matrix is given by
\scalebox{0.7}{$\left(\begin{array}{cc}2 & -1 \\-1 & 2\end{array}\right)$}.  The lattice $A_2$ is even
and of signature $(2,0)$. Moreover, its  discriminant group
is $A_{A_2}\coloneqq A_2^*/A_2\cong\ZZ/3\ZZ$ and, in particular, $A_2$ is  not unimodular. 

Due to \cite[Thm.\ 1.14.4]{NikulinInt}, there exist  embeddings
$$A_2\,\hookrightarrow\Lambda\text{ and }A_2\,\hookrightarrow\widetilde\Lambda,$$
which are both unique up to the action of $\OO(\Lambda)$
and $\OO(\widetilde\Lambda)$, respectively. Note that all such embeddings are automatically primitive.
In the following we will fix once and for all one such embedding $A_2\,\hookrightarrow\Lambda\,\hookrightarrow\widetilde\Lambda$
and consider the orthogonal complement of
$A_2\subset\widetilde\Lambda$ as a fixed primitive sublattice  
\begin{equation}\label{eqn:A2emb}
A_2^\perp\subset\widetilde\Lambda
\end{equation}
of signature $(2,20)$. Its isomorphism type does not depend on the chosen embedding of
$A_2$. It can be described explicitly as the orthogonal complement of
the  embedding $A_2\,\hookrightarrow \widetilde \Lambda$  given by
\begin{equation}\label{eqn:explemb}
A_2\,\hookrightarrow U\oplus U\,\hookrightarrow \widetilde\Lambda,~\lambda_1\mapsto e'+f',~ \lambda_2\mapsto e+f-e',
\end{equation}
where $e,f$ and $e',f'$ denote the standard bases of the two copies of the hyperbolic plane.
Thus,
\begin{equation}\label{eqn:A2perp}A_2^\perp\cong E_8(-1)^{\oplus 2}\oplus U^{\oplus 2}\oplus A_2(-1).\footnote{In \cite{HassettComp} 
the last summand is instead described as a lattice with intersection matrix \scalebox{0.7}{$\left(\begin{array}{cc}-2 & -1 \\-1 & -2\end{array}\right)$}, which is of course isomorphic
to $A_2(-1)$.}
\end{equation}

\begin{remark}\label{rem:OA_2}
For later use we recall that the group of isometries $\OO(A_2)$ of the lattice $A_2$ is isomorphic
to ${\mathfrak S}_3\times \ZZ/2\ZZ$. Here, the Weyl group ${\mathfrak S}_3$ permutes the unit vectors
 $e_i\in {\mathbb R}^3$, where $A_2\,\hookrightarrow {\mathbb R}^3$ via $\lambda_1=e_1-e_2$ and $\lambda_2=e_2-e_3$, and
 the generator of $\ZZ/2\ZZ$ acts by $-{\rm id}$. In fact, ${\mathfrak S}_3\subset\OO(A_2)$ is the kernel of the natural map
$\OO(A_2)\to \OO(A_{A_2})\cong\ZZ/2\ZZ$ (use the aforementioned $A_{A_2}\cong\ZZ/3\ZZ$).
The sign ${\mathfrak S}_3\to\ZZ/2\ZZ$ can be identified with the determinant $\OO(A_2)\to\{\pm1\}$. Thus,
the group of orientation preserving isometries of $A_2$ acting trivially on $A_{A_2}$ is just
${\mathfrak A}_3\cong\ZZ/3\ZZ$, where the generator can be chosen to act
by $\lambda_1\mapsto-\lambda_1-\lambda_2$, $\lambda_2\mapsto\lambda_1$.
\end{remark}
%%%%%%%%%%%%%%%%%%%%%%%%%%%%
\subsection{}\label{sec:cubiclat}
%The lattice $A_2^\perp$ can also be embedded into the middle cohomology of any smooth cubic fourfold.
Next, consider the unique odd, unimodular lattice $${\rm I}_{2,21}\coloneqq\ZZ^{\oplus 2}\oplus\ZZ(-1)^{\oplus 21}\cong E_8(-1)^{\oplus 2}\oplus U^{\oplus 2}\oplus\ZZ(-1)^{\oplus 3}$$ of signature $(2,21)$ and an element $h\in{\rm I}_{2,21}$
with $(h)^2=-3$, e.g.\ $h=(1,1,1)\in \ZZ(-1)^{\oplus 3}$. Then the primitive sublattice $h^\perp\subset{\rm I}_{2,21}$ is of signature $(2,20)$
and using (\ref{eqn:A2perp}) one finds $$h^\perp\cong A_2^\perp.$$

In the following, we will always consider $A_2^\perp$ with two fixed embeddings as above:
$$\xymatrix{{\rm I}_{2,21}&A_2^\perp\ar@{^(->}[r] \ar@{_(->}[l]&\widetilde\Lambda.}$$

Following Hassett \cite{HassettComp}, we now consider all primitive, negative definite sublattices
$$h\in K_d\subset{\rm I}_{2,21}$$ of rank two
containing $h$. Here, the index $d={\rm disc}(K_d)$ denotes the discriminant of $K_d$,
which is necessarily positive. Using \cite[Sec.\ 1.5]{NikulinInt}
one finds that up to the action of the subgroup of $\OO({\rm I}_{2,21})$ fixing $h$ the lattice
$K_d\subset{\rm I}_{2,21}$ is uniquely determined by $d$, see \cite[Prop.\ 3.2.4]{HassettComp} for the details. 

Furthermore, $d\equiv0,2\,(6)$ and the generator  $v$ of $K_d\cap A_2^\perp$ (unique up to sign)
satisfies:
\begin{equation}\label{eqn:v^2}
-(v)^2=\begin{cases}d/3&\text{ if }d\equiv0\,(6)\\
3d&\text{ if }d\equiv2\,(6).
\end{cases}
\end{equation}

More precisely, Hassett shows that up to isometry of $A_2^\perp$ the vector $v$ is given as
\begin{equation}\label{eqn:eplicv}
v=e_1-(d/6)f_1\text{ and }v=3(e_1-((d-2)/6) f_1)+\mu_1-\mu_2,
\end{equation}
respectively. Here, $e_1,f_1$ is the standard basis of one of the copies of $U$ in (\ref{eqn:A2perp}) and $\mu_1,\mu_2$ denotes the standard basis of $A_2(-1)$.
Viewing $v\in A_2^\perp\subset\widetilde\Lambda$ as an element of $\widetilde\Lambda$ leads to a lattice
$$A_2\oplus\ZZ \cdot v\subset\widetilde\Lambda$$ of rank three and
signature $(2,1)$. As it turns out, this is a primitive sublattice for $d\equiv0\, (6)$ and it is  of index three in its saturation for $d\equiv2\,(6)$.
This follows from \cite[Prop.\ 3.2.2]{HassettComp} asserting that in the two cases $(v.A_2^\perp)=\ZZ$ and $3\ZZ$, respectively.
Altogether this yields

\begin{lem}\label{lem:dics11}
The saturation $\Gamma_d\subset\widetilde\Lambda$ of $A_2\oplus\ZZ\cdot v$ satisfies
$${\rm disc}(\Gamma_d)=d.$$
\end{lem}

\begin{proof} This can either be proved by a direct computation or by observing that ${\rm disc}(\Gamma_d)$
equals the discriminant of $\Gamma_d^\perp\subset\widetilde\Lambda$, which is isomorphic
to $v^\perp\subset A_2^\perp$. Similarly, $d={\rm disc}(K_d)$ equals the discriminant of the lattice
$\langle v,h\rangle^\perp$, which again is just $v^\perp\subset A_2^\perp$.
%Indeed, for $d\equiv0\,(6)$ one has ${\rm disc}(\Gamma)={\rm disc}(A_2)\cdot{\rm disc}(\ZZ\cdot v)=3\cdot(d/3)$
%and for $d\equiv 2\, (6)$ the standard formula for the discriminant of finite
%index sublattices shows ${\rm disc}(\Gamma)=(1/3)^2\cdot{\rm disc}( A_2)\cdot{\rm disc}(\ZZ\cdot v)=(1/3)^2\cdot3\cdot(3d)$.
\end{proof}

In our discussion, the lattices $K_d$ will take a back seat, as it will be  more natural to
work with the generator $v\in A_2^\perp\cap K_d$ directly.

%%%%%%%%%%%%%%%%%%%%
\subsection{}\label{sec:defperiod} We shall be interested in the period domains
$$D\subset\PP(A_2^\perp\otimes\CC)\text{ and } Q\subset\PP(\widetilde\Lambda\otimes\CC),$$ defined by the two conditions $(x.x)=0$ and $(x.\bar x)>0$. Observe that
 $$\dim \,D=20\text{ and }\dim\, Q=22$$  and that $Q$ is connected while $D$ has two connected components.
Using the embedding (\ref{eqn:A2emb}), we can 
write $D=\PP(A_2^\perp\otimes\CC)\cap Q$
as part of the  commutative diagram
$$\xymatrix@R-8pt{D\ar@{^(->}[r]\ar@{^(->}[d]&Q\ar@{^(->}[d]\\
\PP(A_2^\perp\otimes\CC)\ar@{^(->}[r]&\PP(\widetilde\Lambda\otimes\CC).}
$$

Thus, points $x\in D$ correspond to Hodge structures of weight two on the lattice $A_2^\perp$, but also to
 Hodge structures on $\widetilde\Lambda$ with $A_2$ contained in the $(1,1)$-part. In fact,
for  very general points $x\in D$  the integral $(1,1)$-part of the corresponding Hodge structure is the lattice $A_2$.

We shall refer to $D$ as the period domain of cubic fourfolds, although only an open subset really corresponds
to smooth cubics. More concretely, for a smooth cubic $X\subset\PP^5$
and any marking, i.e.\ an isometry, $\varphi\colon h^\perp\congpf A_2^\perp$ (up to sign), one defines the associated
period  as the image $x\coloneqq[\varphi_\CC(H^{3,1}(X))]\in D$.
A description of the image of the period map, allowing cubics with ADE-singularities, has been given by
Laza \cite{Laza} and Looijenga in \cite{Loo}. Points in $Q$ are thought of as periods of generalized K3 surfaces, cf.\ Section \ref{sec:GenK3}.

For later use we state the following technical observation.

\begin{lem}\label{lem:revor}
The Hodge structure on $\widetilde\Lambda$ defined by an arbitrary $x\in D$ admits a Hodge isometry that
reverses any given orientation of the four positive directions.
\end{lem}

\begin{proof}
Consider a transposition $g\coloneqq(12)\in{\mathfrak S}_3\subset{\rm O}(A_2)$. Then $g$ acts trivially on the dis\-cri\-minant
$A_{A_2}$ (see Remark \ref{rem:OA_2}) and can, therefore, be extended to $\tilde g\in{\rm O}(\widetilde\Lambda)$ acting trivially on $A_2^\perp$.
Thus, the Hodge structure defined by $x$ admits a Hodge isometry $\tilde g$, which preserves the orientation of the two positive directions given by
 the $(2,0)$ and $(0,2)$-parts. On the other hand, by construction, it reverses the  orientation of the
two positive directions in $A_2$.
\end{proof}

\begin{remark}
This result is the analogue of the observation that any Hodge structure on $\widetilde\Lambda$ containing
a hyperbolic plane $U$ in its $(1,1)$-part admits  an orientation reversing Hodge isometry. This assertion applies to the Hodge structure
$\widetilde H(S,\ZZ)$ of a K3 surface $S$, but it is not clear whether also $\widetilde H(S,\alpha,\ZZ)$ of a twisted K3 surface $(S,\alpha)$ (see below)  admits an orientation reversing Hodge isometry. The latter would be important for
adapting the arguments in \cite{HMS} to the description of the image of $\Aut(\Db(S,\alpha))\to\Aut(\widetilde H(S,\alpha,\ZZ))$, see \cite{Reinecke} for partial results.
\end{remark}
%%%%%%%%%%%%%%%%%%%%%%%%
\subsection{} Let us now turn to the geometric interpretation of certain periods in $Q$.
Recall that for a K3 surface $S$ the extended K3 (or Mukai) lattice $\widetilde H(S,\ZZ)$ is abstractly isomorphic
to $\widetilde \Lambda$. Moreover, $\widetilde H(S,\ZZ)$ comes with a natural Hodge structure of weight two defined
by $$\widetilde H^{2,0}(S)\coloneqq H^{2,0}(S)\text{ and }\widetilde H^{1,1}(S)\coloneqq H^{1,1}(S)\oplus (H^0\oplus H^4)(S,\CC).$$
For a Brauer class $\alpha\in{\rm Br}(S)\cong H^2(S,{\mathbb G}_m)\cong H^2(S,\ko_S^*)_{\rm tors}$ we have introduced
in \cite{HuyInt} the weight-two Hodge structure $\widetilde H(S,\alpha,\ZZ)$. As a lattice
this is still isomorphic to $\widetilde\Lambda$ and its Hodge structure is determined
by
$$\widetilde H^{2,0}(S,\alpha)\coloneqq\CC\cdot(\sigma+B\wedge\sigma)\text{ and } \widetilde H^{1,1}(S,\alpha)\coloneqq
\exp(B)\cdot \widetilde H^{1,1}(S).$$
Here, $0\ne \sigma\in H^{2,0}(S)$ and
$B\in H^2(S,\QQ)$ maps to $\alpha$ under the exponential
map 
$$\xymatrix{H^2(S,\QQ)\ar[r]& H^2(S,\ko_S)\ar[r]^{\exp}& H^2(S,\ko_S^*).}$$
The isomorphism type of the Hodge structure is independent of the choice of $B$.

\begin{definition}
A period $x\in Q$ is of \emph{K3 type} (resp.\ \emph{twisted K3 type}) if there exists a
K3 surface $S$ (resp.\ a twisted K3 surface $(S,\alpha\in{\rm Br}(S))$)
such that the Hodge structure on $\widetilde\Lambda$ defined by $x$ is Hodge isometric to $\widetilde H(S,\ZZ)$
(resp.\ $\widetilde H(S,\alpha,\ZZ)$).  
\end{definition}

The sets of periods of K3 type and twisted K3 type will be denoted
$$Q_{\rm K3}\subset Q_{\rm K3'}\subset Q.$$

There is also a geometric interpretation for points outside $Q_{\rm K3'}$ in terms
of symplectic structures \cite{HuyInt}, but those are a priori inaccessible by algebro-geometric techniques (see however
Section \ref{sec:GenK3}).

For the following recall that the twisted hyperbolic plane $U(n)$ is the rank two lattice with
intersection matrix
\scalebox{0.7}{$\left(\begin{array}{cc}0 & n \\n & 0\end{array}\right)$}.  The standard isotropic generators
will be denoted $e_n,f_n$ or simply $e,f$. Part i) of the next lemma is well known.

\begin{lem}\label{lem:UUn} Consider a period point $x\in Q$. Then:\\
{\rm i)} $x\in Q_{\rm K3}$ if and only if there exists an embedding $U\,\hookrightarrow \widetilde\Lambda$ 
into the $(1,1)$-part of the Hodge structure defined by $x$.\\
{\rm ii)} $x\in Q_{\rm K3'}$ if and only if there exists a (not necessarily  primitive) embedding $U(n)\,\hookrightarrow \widetilde\Lambda$ for some
$n\ne0$ into the $(1,1)$-part of the Hodge structure defined by $x$.
\end{lem}

\begin{proof} We prove the second assertion; the first one is even easier. Start with a twisted K3 surface $(S,\alpha)$ and
pick a lift $B\in H^2(S,\QQ)$ of $\alpha$. Then
the algebraic part $\widetilde H^{1,1}(S,\alpha,\ZZ)=\exp(B)\cdot \widetilde H^{1,1}(S,\QQ)\cap \widetilde H(S,\ZZ)$
contains the  lattice $(\ZZ\cdot (1,B,B^2/2)\cap \widetilde H(S,\ZZ))\oplus H^4(S,\ZZ)$, which is isomorphic to $U(n)$
for $n$ minimal with $n(1,B,B^2/2)\in \widetilde H(S,\ZZ)$.

Conversely, assume $U(n)\subset\widetilde\Lambda$ is of type $(1,1)$ with respect to $x$. Choosing $n$ minimal,
we can assume that the standard isotropic generator $e_n=e$ is primitive in $\widetilde\Lambda$. But then
$e\in U(n)$ can  be completed to a sublattice of $\widetilde\Lambda$ isomorphic to the hyperbolic
plane $U=\langle e,f\rangle$, which therefore induces an orthogonal decomposition (usually different from the one defining $\widetilde\Lambda$)
\begin{equation}\label{eqn:decomp}
\widetilde\Lambda\cong\Lambda\oplus U.
\end{equation}

With respect to (\ref{eqn:decomp})  the second basis vector $f_n\in U(n)$
can be written as $f_n=\gamma+nf+k e$ with $\gamma\in\Lambda$. Similarly,  a  generator of the $(2,0)$-part 
of the Hodge structure determined by $x$ is orthogonal
to $e$ and hence of the form $\sigma+\lambda e$ for some $\sigma\in\Lambda\otimes\CC$ and $\lambda\in\CC$.
However, it is also orthogonal to $f_n$ and so $(\gamma.\sigma)+n\lambda=0$. 
Now set $B:=-(1/n)\gamma$. Then $\sigma+\lambda e=\sigma+B\wedge \sigma$, where $B\wedge\sigma$ stands for $(B.\sigma)e$.

 Eventually, the surjectivity of the period map implies that $\sigma\in\Lambda\otimes\CC$ can be realized as the period of some K3 surface $S$,
i.e.\ there exists an isometry $H^2(S,\ZZ)\cong\Lambda$ identifying $H^{2,0}(S)$ with
$\CC\cdot\sigma\subset\Lambda\otimes\CC$.  Here one uses $(\sigma.\sigma)=(\sigma+\lambda e.\sigma+\lambda e)=0$ and $(\sigma.\bar \sigma)=(\sigma+\lambda e.\bar\sigma+\bar\lambda e)>0$.
Mapping $H^4(S,\ZZ)$ to $\ZZ\cdot e\subset U\subset\Lambda\oplus
U$ in (\ref{eqn:decomp}) and defining $\alpha\in{\rm Br}(S)$ as the Brauer class induced 
by $B$ under $\Lambda\otimes\QQ\cong H^2(S,\QQ)\to H^2(S,{\mathbb G}_m)$ yields a Hodge isometry
between $\widetilde H(S,\alpha,\ZZ)$  and the Hodge structure defined by $x$ on $\widetilde\Lambda$.
\end{proof}

\begin{cor}
The sets $Q_{\rm K3}\subset Q_{\rm K3'}\subset Q$ can be described as the
intersections of $Q$ with countably many linear subspaces of codimension two:
$$Q_{\rm K3}=Q\cap\bigcup U^\perp\subset Q_{\rm K3'}=Q\cap\bigcup U(n)^\perp\subset Q.$$
Here, the first union is over all embeddings $U\,\hookrightarrow \widetilde\Lambda$  and the second over all
 $U(n)\,\hookrightarrow\widetilde\Lambda$ with arbitrary $n\ne0$.\qed
\end{cor}

\subsection{}\label{sec:DK3'codim}
However, it will turn out that the further intersection with  $D$ yields  countable unions of codimension one subsets.
These intersections are denoted by
$$D_{\rm K3}\coloneqq D\cap Q_{\rm K3}\subset D_{\rm K3'}\coloneqq D\cap Q_{\rm K3'}\subset D$$
and will be viewed as the sets of cubic periods that define generalized K3 periods of K3 type and of twisted K3 type,
respectively.
So

$\bullet$ $x\in D_{\rm K3}$ if and only if there exists a K3 surface $S$ such that the Hodge structure on $\widetilde\Lambda$ defined
by $x$ is Hodge isometric to $\widetilde H(S,\ZZ)$.

$\bullet$  $x\in D_{\rm K3'}$ if and only if there exists
a twisted K3 surface $(S,\alpha\in{\rm Br}(S))$ such that the Hodge structure on $\widetilde\Lambda$ defined
by $x$ is Hodge isometric to $\widetilde H(S,\alpha,\ZZ)$.

\smallskip

We remark for later use that for very  general $x\in D_{\rm K3}$ (or $x\in D_{\rm K3'}$) the algebraic part $\widetilde H^{1,1}(S,\ZZ)$ (resp.\
$\widetilde H^{1,1}(S,\alpha,\ZZ)$) is of rank three.
\medskip

We will first explain that $D_{\rm K3'}$ is a countable union of hyperplane sections. A second proof for the same assertion that
also works for $D_{\rm K3}$ is provided in Section \ref{sec:DK3}.

\begin{prop}\label{prop:DtwK3De}
The set of twisted K3 periods in $D$ can also be described as the countable union of hyperplane sections:
$$D_{\rm K3'}=D\cap\bigcup e^\perp.$$
 Here, the union runs over all $0\ne e\in\widetilde\Lambda$ with $(e)^2=0$.
\end{prop}

\begin{proof} One inclusion follows from the fact that any $U(n)$ contains an isotropic vector.
For the converse, assume $x\in e^\perp$ for some primitive isotropic $0\ne e\in \widetilde\Lambda$.
%can be realized as the standard basis vector of a hyperbolic plane $U\subset\widetilde\Lambda$. 
Then $e\not\in A_2^\perp$, as otherwise the positive plane corresponding to $x$ would be
contained in the orthogonal complement of $e$ in $A_2^\perp$, which has only one positive direction left.
Hence, there exists  $a\in A_2$ with
$(a.e)\ne0$. Let $f\coloneqq (a.e)a-((a)^2/2)e$, which satisfies $(f)^2=0$ and
$(f.e)=(a.e)^2\eqqcolon n$. Hence, $e,f$ span  a twisted hyperbolic
plane $U(n)$  in the
$(1,1)$-part of the Hodge structure defined by $x$.
\end{proof}

%Now, for $x\in D_{\rm K3'}$ there exists  a sublattice $U(n)\subset\widetilde \Lambda$
%contained in the $(1,1)$-part of the Hodge structure corresponding to $x$, i.e.\ 
%$x\in D\cap U(n)^\perp$. However, due to the lemma there is another sublattice $U(m)\subset\widetilde\Lambda$ with $x%\in U(m)^\perp$
%and $\rk(A_2+U(m))=3$. Hence, the intersection $D\cap U(n)^\perp\subset D_{\rm K3'}$ of codimension two
%is subsumed by the codimension one set $D\cap U(m)^\perp\subset D_{\rm K3'}$.\footnote{It would be interesting
%to come up with a similar argument proving
%that  also $D_{\rm K3}$ is of codimension one. Instead one has to pass via the description of $D_{\rm K3}$
%in terms of divisors $D_d$ and then use the more involved \cite[Thm.\ 3.1]{AT}. See Proposition \ref{prop:DtwDd}.}
%A better way to phrase this is as follows:

%\begin{remark}\label{rem:decompe}
%Any $0\ne e\in\widetilde \Lambda$ with $(e)^2=0$ can be decomposed as $e=e_1+e_2\in (A_2\oplus A_2^\perp)\otimes%\QQ$.
%Since $A_2$ is positive definite, we must have $e_2\ne0$. However, it can happen that $e_1=0$, but then
%$D\cap e^\perp=\emptyset$, as $A_2^\perp$ has signature
%$(2,20)$ and so $e_2^\perp\subset A_2^\perp$ does not contain any positive plane.
%So classes $e\in\widetilde\Lambda$ with $e_1=0$ are of no interest to us and we can safely 
%ignore them.
%\end{remark}

There is yet another class of hyperplane sections of $D$ that is of importance to us.
We let $$D_{\rm sph}\coloneqq D\cap\bigcup\delta^\perp,$$
where the union is over all $\delta\in \widetilde\Lambda$ with $(\delta)^2=-2$ and call it the
set of \emph{periods with spherical classes}. Indeed, $x\in D$ is contained in $D_{\rm sph}$ if and only if
the  Hodge structure on $\widetilde\Lambda$ defined by $x$ admits an integral $(1,1)$-class
$\delta$ with $(\delta)^2=-2$ and those classes typically appear as Mukai vectors of spherical objects,
see Example \ref{exa:sphtw}.

Note that there are natural inclusions
$$D_{\rm K3}\subset D_{\rm sph}\subset D,$$
for every hyperbolic plane $U$ contains a $(-2)$-class. 
However, $D_{\rm K3'}$ is
not contained in $D_{\rm sph}$ and, more precisely, the inclusions
$$D_{\rm K3} \subsetneq D_{\rm K3'}\cap D_{\rm sph}\subsetneq D_{\rm K3'}, D_{\rm sph}$$
are all proper, see Example \ref{ex:properincl} and Proposition \ref{prop:Dsphnot}.

%%%%%%%%%%%%%%%%%%%%%%%%%%%%%%%
\subsection{}\label{sec:DK3}
It is instructive to study the sets $D_{\rm K3}\subset D_{\rm K3'}$ and $D_{\rm K3}\subset D_{\rm sph}$
from a  more cubic perspective, i.e.\ in terms
of the lattices $K_d$. 

For any $h\in K_d\subset {\rm I}_{2,21}$ as in Section \ref{sec:cubiclat} one introduces
the hyperplane section
$$D\cap K_d^\perp\subset \PP(A_2^\perp\otimes\CC)$$
of all cubic periods orthogonal to $K_d\cap A_2^\perp$. In other words, $D\cap K_d^\perp$ is the set 
of cubic periods for which the  generator $v$ of $K_d\cap A_2^\perp$ is of type $(1,1)$, i.e.\ $D\cap K_d^\perp=D\cap v^\perp$.
Then one defines $$D_d\coloneqq D\cap \bigcup  K_d^\perp,$$
where the union runs over all $h\in K_d\subset {\rm I}_{2,21}$ as above.
So, for each positive $d\equiv0,2\,(6)$ the set $D_d$ is a countable union of hyperplane sections of $D$.
Dividing $D_d$ by the subgroup $\tilde\OO(h^\perp)=\OO({\rm I}_{2,21},h)\subset \OO({\rm I}_{2,21})$ of elements fixing $h$ yields  Hassett's irreducible
divisor $$\kc_d\coloneqq D_d/\tilde\OO(h^\perp).$$
%\medskip

\noindent
Consider the following conditions for an even integer $d>6$:
\smallskip

($\ast$)  $d\equiv0,2\,(6)$.

\smallskip

($\ast$$\ast$) 
$d$ is even and $d/2$ is not divisible by $9$ or any prime
$p\equiv 2\, (3)$
%The integer $d$ is not divisible by $4,9$, or any prime $2\ne p\equiv2\,(3)$.
\smallskip

($\ast$$\ast'$) 
$d$ is even and  in 
the prime factorization $d/2=\prod p_i^{n_i}$ one has $n_i\equiv0\, (2)$ for all primes $p_i\equiv2\, (3)$.
%$n_i\equiv0\,(2)$ for any prime $p_i\equiv2\,(3)$ in the prime factor decomposition $2d=\prod p_i^{n_i}$.
\medskip

\noindent
Obviously, ($\ast$$\ast$) implies ($\ast$$\ast'$). More precisely, if $d$ satisfies ($\ast$$\ast$) then  ($\ast$$\ast'$)
holds for all $k^2d$.

\begin{remark}
Conditions ($\ast$) and ($\ast$$\ast$) have first been introduced and studied by Hassett \cite{HassettComp}.
He shows that $D_d$ is not empty
if and only if ($\ast$) is satisfied. Moreover, $d$ satisfies ($\ast$$\ast$) if and only if
for all cubics $X$ with period $x$ contained in $D_d$ there exists a polarized K3 surface $(S,H)$
such that its primitive cohomology $H^2(S,\ZZ)_{\rm pr}$ is Hodge isometric to the Hodge structure
on $K_d^\perp$ defined by $x$. To get polarized K3 surfaces and not only quasi-polarized ones, one has
to use a result of Voisin \cite[Sec.\ 4, Prop.\ 1]{Voisin} saying that $H^{2,2}(X,\ZZ)_{\rm pr}$ does not contain any
class of square $2$.

On the level of lattices this boils down to the observation that for $v\in A_2^\perp$ as in (\ref{eqn:eplicv}), say for $d\equiv0\,(6)$, 
its orthogonal complement in $A_2^\perp$ is isometric to $E_8(-1)^{\oplus 2}\oplus U\oplus A_2(-1)\oplus\ZZ(d/3)$.
And indeed for $d$ satisfying ($\ast$$\ast'$)  $A_2(-1)\oplus\ZZ(d/3)\cong U\oplus \ZZ(-d)$ (see \cite[Cor.\ 1.10.2, 1.13.3]{NikulinInt} or \cite[Thm.\ 14.1.5]{HK3})
and, therefore,
$v^\perp\cong E_8(-1)^{\oplus 2}\oplus U^{\oplus 2}\oplus \ZZ(-d)$, which is the transcendental lattice of a very general
polarized K3 surface of degree $d$. A similar argument holds for $d\equiv 2\,(6)$.
\end{remark}

\begin{prop}\label{prop:DtwDd}
With the above notations one has
$$D_{\rm K3}=\bigcup_{(\ast\ast)} D_d\text{ and } D_{\rm K3'}=\bigcup_{(\ast\ast')}D_d,$$
where $d$ runs through all $d$ satisfying {\rm ($\ast$$\ast$)} resp.\ {\rm ($\ast$$\ast'$)}.
\end{prop}

\begin{proof}
The first equality is due to Addington--Thomas \cite[Thm.\ 3.1]{AT}. Indeed, they show that $x\in D_d$ with $d$ satisfying
($\ast$$\ast$) if and only if there exists a hyperbolic plane $U\subset\widetilde \Lambda$ which is of type $(1,1)$ with respect to $x$.
The latter is in turn equivalent to $x\in D_{\rm K3}$, see Lemma \ref{lem:UUn}.\footnote{The `only if' direction is
a consequence of Hassett's original result saying that  $x\in D_d$ with $d$ satisfying
($\ast$$\ast$) if and only if there exists a polarized K3 surface $(S,H)$
such that $H^2(S,\ZZ)_{\rm prim}$ is Hodge isometric to the Hodge structure on $K_d^\perp$ 
given by $x$. As the orthogonal complement of $H^2(S,\ZZ)_{\rm prim}\subset\widetilde H(S,\ZZ)\cong\widetilde
\Lambda$ contains a hyperbolic plane, by \cite[Thm.\ 1.14.4]{NikulinInt} this Hodge isometry extends to a Hodge isometry
of $\widetilde H(S,\ZZ)$ with the Hodge structure on $\widetilde\Lambda$ given by $x$. For the other direction
one has to show that any Hodge isometry between $\widetilde H(S,\ZZ)$ and the one on $\widetilde\Lambda$ given by
$x$ can be used to get a Hodge isometry between the Hodge structure on $K_d^\perp\cap A_2^\perp\subset\widetilde \Lambda$ and $H^2(S,\ZZ)_{\rm prim}$ for some polarization on $S$.} 

Maybe surprisingly, the second assertion is easier to prove. We include the elementary argument.
Due to Corollary \ref{prop:DtwK3De} we know $D_{\rm K3'}=D\cap\bigcup e^\perp$ with
$0\ne e\in\widetilde\Lambda$ isotropic. So for one inclusion one has to 
show that each $D\cap e^\perp$ is of the form $D_d$ with $d$ satisfying ($\ast$$\ast'$).
Decompose $e=e_1+e_2\in (A_2\oplus A_2^\perp)\otimes\QQ$.
Let then $v\in A_2^\perp$ such that $\QQ\cdot e_2\cap A_2^\perp=\ZZ\cdot v$
and define $K_d\subset {\rm I}_{2,21}$ as the saturation of the sublattice spanned by $v\in A_2^\perp\subset{\rm I}_{2,21}$
and $h$. We have to show that the discriminant $d$ of $K_d$ satisfies ($\ast$$\ast'$).

Assume first that $A_2\oplus \ZZ \cdot v\subset \widetilde\Lambda$ is primitive. Then
$d\equiv0\,(6)$ and $d=-3(v)^2$, see Section \ref{sec:cubiclat}. As $e\in A_2\oplus\ZZ\cdot v$ in this case,
the quadratic equation $2(x_1^2+x_2^2-x_1x_2)+(v)^2x^2=0$ admits an integral solution.
However, it is a classical result that 
\begin{equation}\label{eqn:class}
2n=(w)^2
\end{equation} for some $w\in A_2$ if and only if $n=\prod p_i^{n_i}$
with $n_i\equiv0\, (2)$ for all $p_i\equiv 2\, (3)$, see \cite{Kneser}.\footnote{For example, a prime $p$
can be written as $x^2+3y^2$ if and only if $p=3$ or $p\equiv1\, (3)$, see \cite{Cox}. Since $4(x^2+xy+y^2)=(2x+y)^2+3y^2$
and $(x_1^2+3y_1^2)\cdot(x_2^2+3y_2^2)=(x_1x_2-3y_1y_2)^2+3(x_1y_2+x_2y_1)^2$,
this proves one direction. The other one uses a computation with
Hilbert symbols to determine when $-nx_1^2+x_2^2+3x_3^2=0$ has a rational solution.}
But clearly this holds for $n=-(v)^2/2$ if and only if $d=6n$ satisfies ($\ast$$\ast'$).

Next assume that $A_2\oplus \ZZ \cdot v\subset \widetilde\Lambda$  has index three in its saturation.
Hence, $d\equiv 2\, (6)$ and $3d=-(v)^2$. Then argue as before, but now with the isotropic vector
$3e\in A_2\oplus\ZZ\cdot v$ and with $n=-(v)^2/2=3d/2$.

Running the argument backwards proves the reverse inclusion.
\end{proof}

So in particular, although $Q_{\rm K3}\subset Q_{\rm K3'}\subset Q$ are countable unions of codimension two
subsets, their restrictions $D_{\rm K3}\subset D_{\rm K3'}\subset D$ to $D$ are countable unions
of codimension one subsets. For $D_{\rm K3'}$ we have observed this already in Section \ref{sec:DK3'codim}.

\begin{remark}
As mentioned in \cite{AT,Add} and explained to me by Addington, condition ($\ast$$\ast$) is in fact equivalent to
the existence of a primitive $w\in A_2$ with $d=(w)^2$. And, as has become clear in the
above proof, condition ($\ast$$\ast'$) is equivalent to the existence of a (not necessarily primitive) $w\in A_2$
with $d=(w)^2$.
\end{remark}

The first values of $d>6$ that satisfy the various conditions  are 

\begin{center}
 \begin{tabular}[t]{|c|c|c|c|c|c|c|c|c|c|c|c|c|c|c|c|} \hline
 \raisebox{-0.03cm}{($\ast$)}&\raisebox{-0.03cm}{8}&\raisebox{-0.03cm}{12}&\raisebox{-0.03cm}{14}&\raisebox{-0.03cm}{18}&\raisebox{-0.03cm}{20} &\raisebox{-0.03cm}{24}&\raisebox{-0.03cm}{26}& \raisebox{-0.03cm}{30}&\raisebox{-0.03cm}{32}&\raisebox{-0.03cm}{36}&\raisebox{-0.03cm}{38}&\raisebox{-0.03cm}{42} &\raisebox{-0.03cm}{44}&\raisebox{-0.03cm}{48}\\
[1pt]\hline \raisebox{-0.03cm}{($\ast$$\ast$)} &\raisebox{-0.03cm}{}&\raisebox{-0.03cm}{}&\raisebox{-0.03cm}{14}&\raisebox{-0.03cm}{}&\raisebox{-0.03cm}{}&\raisebox{-0.03cm}{}&\raisebox{-0.03cm}{26}&\raisebox{-0.03cm}{}&\raisebox{-0.03cm}{}&\raisebox{-0.03cm}{}&\raisebox{-0.03cm}{38}&\raisebox{-0.03cm}{42} &\raisebox{-0.03cm}{}&\raisebox{-0.03cm}{}\\
                 [1pt]\hline
                 \raisebox{-0.03cm}{($\ast$$\ast'$)}&\raisebox{-0.03cm}{8}&\raisebox{-0.03cm}{}&\raisebox{-0.03cm}{14}&\raisebox{-0.03cm}{18}&\raisebox{-0.03cm}{} &\raisebox{-0.03cm}{24}&\raisebox{-0.03cm}{26}& \raisebox{-0.03cm}{}&\raisebox{-0.03cm}{32}&\raisebox{-0.03cm}{}&\raisebox{-0.03cm}{38}&\raisebox{-0.03cm}{42} &\raisebox{-0.03cm}{}&\raisebox{-0.03cm}{}\\
                 [1pt]\hline
  \end{tabular}
\end{center}

\bigskip

%So $12$ and $20$ are the first two values which are not even twisted.
\begin{ex} 
For certain $d$ the condition that the period $x\in D$ of  a cubic $X$ is contained in $D_d$ has a geometric interpretation, 
see \cite[Sec.\ 4]{HassettComp}. For example, $x\in D_8$ if and only if $X$ contains a plane $\PP^2\subset X$,
or if $X$ is a Pfaffian cubic, then $x\in D_{14}$.
\end{ex}

Let $x\in D_d$ with $d$ satisfying ($\ast$$\ast'$). Then there exists a twisted K3 surface $(S,\alpha)$ such that
the Hodge structure defined by $x$ is Hodge isometric to $\widetilde H(S,\alpha,\ZZ)$.  If $x\in D_d$ is a very general point of $D_d$, then $\rk(\widetilde H^{1,1}(S,\alpha,\ZZ))=3$ and
$A_2\oplus \ZZ\cdot v\subset\widetilde H^{1,1}(S,\alpha,\ZZ)$ is of index one or three, respectively.

\begin{lem}\label{lem:OrderBrauer}
For the order of the Brauer class $\alpha$ one has
$${\rm ord}(\alpha)^2\mid d.$$
\end{lem}

\begin{proof}
Let $\ell\coloneqq {\rm ord}(\alpha)$. As proved in \cite{HuyInt}, the transcendental lattice
$T(S,\alpha)$ is isometric to the kernel of the natural map $T(S)\twoheadrightarrow (1/\ell)\ZZ/\ZZ$ defined by $\alpha$.
Hence,
$$|{\rm disc}(T(S,\alpha))|=|{\rm disc}(T(S))|\cdot {\rm ord}(\alpha)^2.$$
On the other hand, ${\rm disc}(T(S,\alpha))={\rm disc}(\widetilde H^{1,1}(S,\alpha,\ZZ))=d$ by Lemma 
\ref{lem:dics11}.
\end{proof}

Clearly, $d/{\rm ord}(\alpha)^2$ still satisfies ($\ast$$\ast'$) (but not necessarily ($\ast$$\ast$)). As
mentioned earlier, any $d$ satisfying ($\ast$$\ast'$) can be written (not always uniquely) as
$d=k^2\cdot d_0$ with $d_0$ satisfying ($\ast$$\ast$). For any such factorization one can
indeed choose $(S,\alpha)$ as above such that in addition ${\rm ord}(\alpha)=k$. In particular, then the untwisted
Hodge structure $\widetilde H(S,\ZZ)$ defines a point in $D_{d_0}$. This is best seen by starting with $D_{d_0}$ and then 
choosing globally a B-field which for the very general $S$ in $D_{d_0}$ defines a Brauer class of order $k$.

\subsection{} We will need to say a few things about the spherical locus $D_{\rm sph}$, as this will be crucial later.
\begin{ex}\label{ex:properincl}
i) Consider $d=24$ which obviously satisfies ($\ast$$\ast'$) but not ($\ast\ast$), i.e.\ $D_d\subset D_{\rm K3'}$ but $D_d\not\subset D_{\rm K3}$.
Also, $D_d\subset D_{\rm sph}$. Indeed, if $v$ generates $A_2^\perp\cap K_d$, then $(v)^2=-8$
and hence there exists $\delta\in A_2\oplus \ZZ\cdot v$ with $(\delta)^2=-2$, e.g.\ $2\lambda_1+\lambda_2+v$.
So, as mentioned before, one has a proper inclusion
$$D_{\K3}\subsetneq D_{\rm K3'}\cap D_{\rm sph}.$$

ii) Consider $d=12$ which does not satisfy   ($\ast$$\ast'$). So, $D_{12}\not\subset D_{\rm K3'}$, but
$D_{12}\subset D_{\rm sph}$. Indeed, in this case $v$ in (\ref{eqn:v^2}) satisfies $(v)^2=-4$ and, therefore,
$(\lambda_i+v)^2=-2$. Hence,
$$D_{\rm sph}\not\subset D_{\rm K3'}.$$
\end{ex}

 It would be interesting to find a numerical condition ($\dag$)
such that $D_{\rm sph}=\bigcup D_d$ with the union over all $d$ satisfying ($\dag$).
The best we have to offer at this time is the following

\begin{prop}\label{prop:Dsphnot}
Assume $D_d\subset D_{\rm K3'}$ and $9|d$. Then $D_d\not\subset D_{\rm sph}$.
\end{prop}

\begin{proof} Consider a fixed $K_d$ and the corresponding generator $v$ of $K_d\cap A_2^\perp$. As
$9|d$, clearly $d\equiv 0\, (6)$ and so $A_2\oplus \ZZ\cdot v\subset\widetilde\Lambda$ is primitive. If there were a $(-2)$-class
$\delta\in\widetilde\Lambda$ with $D\cap K_d^\perp=D\cap v^\perp\subset D\cap\delta^\perp$, then $\delta\in A_2+\ZZ\cdot v$
and so $\delta =w+kv$ for some $w\in A_2$ and $k\in \ZZ$. But then $-2=(w)^2-k^2d/3$.
However, if $9|d$, then $k^2d/3\equiv0\,(3)$ and hence $(w)^2=2m$ with $m\equiv 2\,(3)$,
which contradicts (\ref{eqn:class}).
\end{proof}

The following immediate consequence is crucial for the proof of Theorem \ref{thm:noFMvg}, see Section \ref{sec:ProofAut}.
\begin{cor}\label{cor:locustw}
The locus of twisted K3 periods $D_{\rm K3'}$ contains infinitely many hyperplane sections $D_d$ with
$D_d\not\subset D_{\rm sph}$.\qed
\end{cor}

%%%%%%%%%%%%%%%%%%%%%

%%%%%%%%%%%%%%%%%%%%%%%%%%%%%

\subsection{}\label{sec:GenK3}
 In \cite{HuyInt} we have shown that points in $Q$ can be understood as periods of generalized K3 surfaces.
 It is useful to distinguish three types:\footnote{The discussion has been prompted by a question of Ben Bakker.}
\smallskip

{\bf i)} Periods of ordinary K3 surfaces are parametrized by $Q_{\rm K3}$.
Up to the action of ${\rm O}(\widetilde \Lambda)$, the set of these periods
is the intersection of $Q$ with the linear codimension two  subspace $\PP(\Lambda\otimes\CC)\subset\PP(\widetilde\Lambda\otimes\CC)$.
\smallskip

{\bf ii)} More generally, one can consider periods of the form $\sigma+B\wedge\sigma$, where $\sigma\in\Lambda\otimes \CC$ is an ordinary period and $B\in\Lambda\otimes\CC$ (but not necessarily $B\in\Lambda\otimes\QQ$). Up to the action of ${\rm O}(\widetilde\Lambda)$,
these periods are parametrized by the intersection of $Q$ with the linear 
subspace of codimension one $\PP((\Lambda\oplus  \ZZ\cdot f)\otimes\CC)\subset\PP(\widetilde\Lambda\otimes\CC)$.
Here, $f$ is viewed as the generator
of $H^4$.  Note that by definition $Q_{\rm K3'}$ is the subset of periods for which
$B$ can be chosen in $\Lambda\otimes\QQ$. 
\smallskip

{\bf iii)} Periods of the form $\exp(B+i\omega)=1+(B+i\omega)+((B^2-\omega^2)/2+(B.\omega)i)$
are geometrically interpreted as periods associated with complexified symplectic forms. Here, the first and third summands are
considered in $U\cong H^0\oplus H^4$. Periods of this type are parametrized by an open dense subset of $Q$.
\smallskip

In particular, all cubic periods parametrized by $D\subset Q$ should have an interpretation
in terms of these three types. This has been discussed above for type i) and has led to consider
the intersection $D_{\rm K3}=D\cap Q_{\rm K3}$. For type ii) with $B$ rational the intersection
with the cubic period domain gives $D_{\rm K3'}$. It is now natural to ask
whether the remaining periods, so the periods in $D\setminus D_{\rm K3'}$,
are of type ii) with $B$ not rational or rather of type iii), i.e.\ related to complexified symplectic forms.
It is the latter, as shown by the following

\begin{prop}\label{prop:genK3allD}
The Hodge structure of a cubic period $x\in D$ is Hodge isometric to the Hodge structure
of a twisted projective K3 surface $(S,\alpha)$, i.e.\ $x\in D_{
\rm K3'}$,
or to the Hodge structure associated with $\exp(B+i\omega)$.
Furthermore, if the Hodge structure of $x$ is Hodge isometric to a Hodge
structure of the type $\sigma+B\wedge\sigma$, then $B$ can be chosen rational.
\end{prop}

\begin{proof} One first observes that, analogously to Lemma \ref{lem:UUn}, ii),
a period $x\in Q$ is of the type ii) if and only if the integral
$(1,1)$-part of the Hodge structure associated with $x$ contains an isotropic direction.
Indeed, if $x$ is of type ii), i.e.\ of the form $\sigma+B\wedge \sigma$,
then $H^4$ provides an isotropic direction of type $(1,1)$. For the
converse use that any isotropic direction can be completed to a hyperbolic
plane $U$ as a direct summand of $\widetilde\Lambda$.
Now regard $U$  as $H^0\oplus H^4$
with $H^4$ as the given isotropic direction, which is of type $(1,1)$. Hence,  $x$ is indeed of type ii).

Now let $x\in D\cap Q$ be of type ii). It is enough to show that then $x\in D_{\rm K3'}$.
The integral $(1,1)$-part of the Hodge structure associated with $x$ contains
$A_2$ and an isotropic direction,  say $\ZZ\cdot f$.  Then conclude by Proposition \ref{prop:DtwK3De}.
%
%Moreover, $A_2$ and $\ZZ\cdot f$ cannot be orthogonal,
%as otherwise the orthogonal complement $(A_2\oplus\ZZ\cdot f)^\perp$ in $\widetilde\Lambda$ has only one positive direction
%and could, therefore, not accommodate for $x$. But then the arguments of part i) of  the proof
%of Lemma \ref{lem:UnUm} show that
%$A_2\oplus\ZZ\cdot f$ contains some $U(n)$ and hence $x\in Q_{\rm K3'}$
%by Lemma \ref{lem:UUn}, ii). 
%
%View $(H^0\oplus H^4)(S,\ZZ)$ as a hyperbolic plane
%with standard basis $e\in H^0$ and $f\in H^4$. Furthermore, pick one of the
%hyperbolic planes $U_0\subset\Lambda\cong H^2(S,\ZZ)$ with a standard basis $e',f'$
%and denote its orthogonal complement in $\Lambda$ by $\Gamma\subset\Lambda$.
%A period  $x$ of type ii), so of the form $\sigma+B\wedge \sigma$, can then be written
%as $\alpha+\gamma+\lambda f$ with $\alpha\in U_0\otimes\CC$, $\gamma\in\Gamma\otimes\CC$, and 
%$\lambda\in \CC$. If $\lambda=0$, then $x\in Q_{\rm K3}$. If $\lambda\ne0$, then after scaling
%we can assume $\lambda=1$.
%
%Now choose an embedding $A_2\,\hookrightarrow (H^0\oplus H^4)\oplus U_0$ as in
%(\ref{eqn:explemb}). Recall that the embedding is unique up to the action of
%${\rm O}(\widetilde\Lambda)$.\footnote{Check again that the argument is really conclusive.}
%Then the $\QQ$-vector space generated by the intersection of $A_2^\perp$ with $(H^0\oplus H^4)\oplus U_0$
%is spanned by
%$-e'+f'+f$ and $-e+f$. If now $x\in D$, i.e.\ $\alpha+\gamma+f\in A_2^\perp\otimes\CC$, then
%$\alpha=-e'+f'+f$ and hence for $B'\coloneqq e'$ one has
%$\sigma+B\wedge \sigma=\alpha+\gamma+f=\sigma+ B'\wedge\sigma$.
\end{proof}

Note that both subsets, $$D\subset Q\text{ and }Q_{\rm K3}\subset Q,$$ are of codimension two and that they both
parametrize periods that can be interpreted in complex geometric terms (in contrast to
the `symplectic periods' of the form $\exp(B+i\omega)$).  In fact, periods
in $D$ are even algebro-geometric in the sense that essentially all of them are associated with
cubic fourfolds $X\subset\PP^5$, whereas most K3 surfaces are of course not projective.
%It would be interesting to exhibit other subsets (of codimension two) of periods in $Q$ 
%with complex geometric interpretations.

In  categorical language one would want to interpret the inclusion $D\subset Q$
for points in the complement of $D_{\rm K3'}$ as saying that
the cubic K3 category $\ka_X$ associated with the cubic fourfold $X\subset\PP^5$ corresponding
to $x\in D\setminus D_{\rm K3'}$ is equivalent to the derived Fukaya category
${\rm D}{\rm Fuk}(B+i\omega)$ associated with a complexified symplectic form $B+i\omega$.
Deciding which symplectic structures occur here is in principle possible, but establishing
an equivalence $$\ka_X\cong {\rm D}{\rm Fuk}(B+i\omega)$$ will be difficult even in special cases.

The categorical interpretation of $D_{\rm K3}\subset Q$ is the content of \cite{AT}, where it is proved that
at least for a Zariski open dense set of periods $x\in D_{\rm K3}$ the cubic K3 category $\ka_X$ really is equivalent to
$\Db(S)$ of the K3 surface $S$ realizing the Hodge structure associated with $x$.
This paper deals with the categorical interpretation of $D_{\rm K3'}\subset Q$.

\begin{remark}
The period domain $Q\subset \PP(\widetilde\Lambda\otimes\CC)$ contains $D\subset Q$ as a codimension
two subset, but it also contains natural codimension one subspaces. For example, for a K3 surface $S$ and
the Mukai vector $v=(1,0,1-n)\in\widetilde H^{1,1}(S,\ZZ)$ the hyperplane section
$Q\cap v^\perp$ can be seen as the period domain for deformations of the Hilbert scheme $S^{[n]}$. Note, however, that
from a categorical point of view the situation is different, even when one restricts to the codimension two
part that corresponds to projective deformations of the Hilbert scheme. In \cite{MaMe} it is explained how
the non-full subcategory $\Db(S)\subset \Db(S^{[n]})$ deforms sideways.
\end{remark}
 
%%%%%%%%%%%%%%%%%%%%%%%%%%%%%%%%%%%%%%

\section{The cubic K3 category}\label{sec:AX}

Let $X\subset\PP^5$ be a smooth cubic hypersurface. The \emph{cubic K3 category}
associated with $X$ is the category
$$\ka_X\coloneqq\langle\ko_X,\ko_X(1),\ko_X(2)\rangle^\perp\coloneqq\{E\in\Db(X)~|~\Hom(\ko_X(i),E[\ast])=0\text{ for }i=0,1,2\}.$$
The category has first been studied by Kuznetsov in \cite{Kuz1}, see also the more recent \cite{Kuznew}. It
behaves in many respects like the derived category $\Db(S)$ of a K3 surface $S$.
In particular, the double shift $E\mapsto E[2]$ defines a Serre functor  of $\ka_X$ (see \cite{Kuz2}, \cite[Cor.\ 4.3]{Kuz3} and \cite[Rem. 4.2]{KM2}) and the dimension of  Hochschild homology of $\ka_X$ and of $\Db(S)$ coincide, see \cite{Kuz1,Kuz2}.

\begin{ex}\label{ex:kuz} Due to the work of Kuznetsov \cite{Kuz1,Kuz5}, certain cubic K3 categories
$\ka_X$ are known to be equivalent to  bounded derived categories
$\Db(S,\alpha)$ of  twisted K3 surfaces $(S,\alpha)$. For example, if the period $x\in D$ of a cubic $X$
is contained in $D_8$, then $X$ contains a plane and for generic choices there exists a twisted K3 surface $(S,\alpha)$
with $\ka_X\cong \Db(S,\alpha)$. Similarly, if  $X$ is a Pfaffian cubic and hence $x\in D_{14}$, then $\ka_X\cong\Db(S)$ for
the K3 surface $S$  naturally associated with the Pfaffian $X$.
\end{ex}

\begin{remark}
Despite the almost perfect analogy between the cubic K3 category $\ka_X$ and the derived category $\Db(S)$ of
K3 surfaces, certain fundamental issues are more difficult for $\ka_X$.  For example, to the best of my knowledge
no $\ka_X$, which is not equivalent to the derived category $\Db(S,\alpha)$ of some twisted K3 surface $(S,\alpha)$,
has yet been endowed with a bounded t-structure, let alone a stability condition.
See \cite{Toda1, Toda2} for a discussion of special stability conditions on certain $\ka_X$ of the form $\Db(S,\alpha)$.
\end{remark}

The semi-orthogonal decomposition $\Db(X)=\langle\ka_X,{} ^\perp\!\ka_X\rangle$
with ${}^\perp\!\ka_X=\langle\ko_X,\ko_X(1),\ko_X(2)\rangle$ comes with
the full embedding $i_*\colon\ka_X\,\hookrightarrow \Db(X)$ (which is often suppressed in the notation)
and the left and right adjoint
functors $i^*,i^!\colon\Db(X)\to\ka_X$, see \cite[Sec.\ 3]{Kuz2} for a survey. 
%In particular, for all $E\in\Db(X)$ there are exact triangles $$E'\to E\to i_*i^*E\text{ and }i_*i^! E\to E\to E''$$
%with $E'\in {}^\perp\!\ka_X$ and $E''\in\ka_X^\perp$.
%The adjoint functors are related by functorial isomorphisms $$i^*E\cong i^!(E\otimes \omega_X)[2].$$
%Moreover, according to \cite{Kuz4},
%the compositions $$i_*\circ i^*,i_*\circ i^!\colon\Db(X)\to\ka_X\,\hookrightarrow\Db(X)$$ are of Fourier--Mukai type.
%More precisely, $\Db(X)=\langle\ka_X,{}^\perp\!\ka_X\rangle$
%induces a certain semi-orthogonal decomposition  $\Db(X\times X)=\langle\Db(X)\boxtimes\ka_X,\Db(X)\boxtimes\, {}^\perp\!%\ka_X\rangle$
%and  the $\Db(X)\boxtimes\ka_X$-part of $\ko_\Delta$, say $\kp\in\Db(X)\boxtimes\ka_X$, is the Fourier--Mukai kernel
%for $i_*\circ i^*$, i.e.\
%there exists an exact triangle $\kq\to \ko_\Delta\to \kp$ with $\kq\in\Db(X\times X)$ contained in the category
%generated by objects of the form $F\boxtimes\ko_X(i)$, with $F\in\Db(X)$ and $i=0,1,2$.

%%%%%%%%%%%%%%%
\subsection{} For a K3 surface $S$ the Mukai lattice $\widetilde H(S,\ZZ)$ is endowed
with the Hodge structure determined by $\widetilde H^{2,0}(S)=H^{2,0}(S)$ and 
by requiring  $\widetilde H^{2,0}\perp\widetilde H^{1,1}$.
Using the natural isomorphism $K_{\rm top}(S)\cong H^*(S,\ZZ)$ this Hodge structure
can also be regarded as a Hodge structure on $K_{\rm top}(S)$.

In \cite{AT} Addington and Thomas introduce a similar Hodge structure associated with the category $\ka_X$,
defined on $K_{\rm top}(\ka_X)$ and denoted by
$\widetilde H(\ka_X,\ZZ)$. Here, $K_{\rm top}(\ka_X)\subset K_{\rm top}(X)$ is the orthogonal complement
of $\{[\ko], [\ko(1)], [\ko(2)]\}$ with respect to the pairing $\chi(\alpha,\beta)=
\langle v(\alpha),v(\beta)\rangle$ defined in terms of the Mukai vector $v\colon K_{\rm top}(X)\to H^*(X,\QQ)$
and the Mukai pairing on $H^*(X,\QQ)$. It is not difficult to see that one has in fact a semi-orthogonal direct
sum decomposition $$K_{\rm top}(X)=K_{\rm top}(\ka_X)\oplus\langle[\ko_X],[\ko_X(1)],[\ko_X(2)]\rangle.$$
As $H^*(X,\ZZ)$ is torsion free, $K_{\rm top}(X)$ and
$$\widetilde H(\ka_X,\ZZ)\coloneqq K_{\rm top}(\ka_X)$$ are as well. 
The Hodge structure is then defined by $\widetilde H^{2,0}(\ka_X)\coloneqq v^{-1}(H^{3,1}(X))$
and the condition $\widetilde H^{2,0}\perp\widetilde H^{1,1}$.
Furthermore, $N(\ka_X)$ and the transcendental lattice $T(\ka_X)$ of $\ka_X$
are introduced in terms of this Hodge structure
as $\widetilde H^{1,1}(\ka_X,\ZZ)$ and its orthogonal complement $\widetilde H^{1,1}(\ka_X,\ZZ)^\perp$, respectively.
As a lattice $\widetilde H(\ka_X,\ZZ)$ is independent of $X$ and by \cite{AT} any equi\-valence $\ka_X\cong\Db(S)$ (see  Example \ref{ex:kuz})
induces a Hodge isometry $\widetilde H(\ka_X,\ZZ)\cong\widetilde H(S,\ZZ)$
(cf.\ Proposition \ref{prop:indHodge}). In particular, $\widetilde H(\ka_X,\ZZ)$ for all smooth cubics is
abstractly isomorphic to $\widetilde\Lambda$.

%\begin{remark}
%Via the fully faithful embedding $\ka_X\,\hookrightarrow\Db(X)$ the category $\ka_X$ inherits a dg-enhancement which in turn
%can be used to define a rational Hodge structure, see \cite{Blanc}. However, this a priori still depends on the embedding
%$\ka\,\hookrightarrow\Db(X)$ and so does not really lead to a more conceptual definition of $\widetilde H(\ka_X,\ZZ)$.
%\end{remark}

As explained in \cite[Prop.\ 2.3]{AT}, the classes $\lambda_j\coloneqq [i^*\ko_\ell(j)]\in \widetilde H^{1,1}(\ka_X,\ZZ)$, for a line
$\ell\subset X$ and $j=1,2$,
can be viewed as the standard generators of a lattice $A_2\subset\widetilde H^{1,1}(\ka_X,\ZZ)$. Moreover, the Mukai
vector $v\colon \widetilde H(\ka_X,\ZZ)=K_{\rm top}(\ka_X)\to H^*(X,\QQ)$ induces an isometry (up to sign)
$$\langle\lambda_1,\lambda_2\rangle^\perp\congpf h^\perp=H^4(X,\ZZ)_{\rm prim}.$$
In particular, any marking $\varphi\colon h^\perp\congpf A_2^\perp$ induces a marking
$\langle\lambda_1,\lambda_2\rangle\oplus\langle\lambda_1,\lambda_2\rangle^\perp\congpf A_2\oplus A_2^\perp$
and further  a marking
\begin{equation}\label{eqn:markingA}
\widetilde H(\ka_X,\ZZ)\congpf\widetilde\Lambda.
\end{equation}

Conversely, any marking (\ref{eqn:markingA}) inducing  the standard identification $\langle\lambda_1,\lambda_2\rangle\congpf A_2$ yields
a marking $H^4(X,\ZZ)_{\rm prim}\congpf A_2^\perp$. In this sense, (an open set of) points $x\in D$ will be considered as periods of
cubic K3 categories $\ka_X$ via their Hodge structures $\widetilde H(\ka_X,\ZZ)$.

Note that the positive directions of $\widetilde H(\ka_X,\ZZ)$ come with a natural orientation, given by the real and imaginary
parts of $\widetilde H^{2,0}(\ka_X)$ and the oriented basis $\lambda_1,\lambda_2$ of $A_2\subset\widetilde H^{1,1}(\ka_X,\ZZ)$.

\subsection{} As we are also interested in equivalences  $\Db(S,\alpha)\congpf\ka_X$, we collect a few relevant facts dealing with
the topological K-theory of twisted K3 surfaces $(S,\alpha)$. As it turns out, the topological setting
does not require any substantially new arguments. In order to speak of twisted sheaves or bundles, let us fix  a class  $B\in H^2(S,\QQ)$ which  under the exponential map
$H^2(S,\QQ)\to H^2(S,\ko_S^*)$ is mapped to $\alpha$. Next choose a \v{C}ech representative $\{B_{ijk}\}$ of $B\in H^2(S,\QQ)$
and consider the associated \v{C}ech representative  $\{\alpha_{ijk}\coloneqq \exp(B_{ijk})\}$ of $\alpha$.
This allows one to speak of $\{\alpha_{ijk}\}$-twisted sheaves and bundles, in the holomorphic as well as in the topological setting.

As explained in \cite[Prop.\ 1.2]{HuSt}, any $\{\alpha_{ijk}\}$-twisted bundle $E$
can be `untwisted' to a bundle $E_B$ by changing the transition functions $\varphi_{ij}$ of $E$ to $\exp(a_{ij})\cdot\varphi_{ij}$, where the continuous functions $a_{ij}$ satisfy $-a_{ij}+a_{ik}-a_{jk}=B_{ijk}$. The process can be reversed and so
the categories of  $\{\alpha_{ijk}\}$-twisted topological bundles is equivalent to the category of untwisted topological bundles. In
particular, $$K_{\rm top}(S,\alpha)\cong K_{\rm top}(S)$$ which composed with the Mukai vector
yields an isomorphism $K_{\rm top}(S,\alpha)\cong \widetilde H(S,\alpha,\ZZ)$ that identifies
the image of $K(S,\alpha)\to K_{\rm top}(S,\alpha)$ with $\widetilde H^{1,1}(S,\alpha,\ZZ)$.
\smallskip

The next result is the twisted version of the observation by Addington and Thomas mentioned earlier.
\begin{prop}\label{prop:indHodge}
Any linear, exact equivalence $\ka_X\cong \Db(S,\alpha)$ induces
a Hodge isometry
$$\widetilde H(\ka_X,\ZZ)\cong\widetilde H(S,\alpha,\ZZ).$$
\end{prop}

\begin{proof}
By results due to Orlov in the untwisted case and to Canonaco and Stellari \cite{CanSte} in the twisted case,
%combined with the fact that $\Db(S,\alpha)$ is saturated due to  Bondal and van den Bergh
 %\cite[Thm.\ 5.14]{HuyFM},\footnote{Saturated in twisted case? Check!}
 any fully faithful functor $\Phi\colon\Db(S,\alpha)\to\Db(X)$ is of Fourier--Mukai
type, i.e.\ $\Phi\cong\Phi_\ke$ for some $\ke\in\Db(S\times X,\alpha^{-1}\boxtimes1)$.
Therefore, $\Phi$ induces a homomorphism
$\Phi^K_\ke\colon K_{\rm top}(S,\alpha)\to K_{\rm top}(X)$, see \cite[Rem.\ 3.4]{HivdB}.

If $\Phi$ is induced by an equivalence $\Db(S,\alpha)\congpf\ka_X$, then
$\Phi^K_\ke\colon K_{\rm top}(S,\alpha)\congpf K_{\rm top}(\ka_X)$ is an isomorphism
and in fact a Hodge isometry $\widetilde H(S,\alpha,\ZZ)\congpf \widetilde H(\ka_X,\ZZ)$.
The compatibility with the Hodge structure follows from the twisted Chern character
${\rm ch}^{-\alpha\boxtimes 1}(\ke)$ of the Mukai kernel being  of Hodge type.
See \cite[Sec.\ 1]{HuSt} for the notion of twisted Chern characters.
That the quadratic form is respected as well is proved by mimicking the argument for  FM-equivalences, see e.g.\ \cite[Sec.\ 5.2]{HuyFM}.

(We are suppressing a number of technical details here. As explained before,
the actual realization of the  Hodge structure $\widetilde H(S,\alpha,\ZZ)$ depends on the choice of a $B\in H^2(S,\QQ)$
lifting $\alpha$. Similarly, the Chern character ${\rm ch}^{-\alpha\boxtimes 1}(\ke)$ also actually depends on $B$.)
\end{proof}

%%%%%%%%%%%%%%%%%%%
\subsection{}  The above result generalizes to \emph{FM-equivalences} $\ka_X\congpf\ka_{X'}$, i.e.\ to equivalences
for which the composition $\Db(X)\to\ka_X\congpf\ka_{X'}\to\Db(X')$ admits  a Fourier--Mukai kernel. It has been conjectured that in
fact any linear exact equivalence is a FM-equivalence, but the existing results do not cover our case.

\begin{prop}\label{prop:equivHodgeA}
Any FM-equivalence $\ka_X\congpf\ka_{X'}$ induces a Hodge isometry
$$\widetilde H(\ka_X,\ZZ)\congpf\widetilde H(\ka_{X'},\ZZ).$$
\end{prop}

\begin{proof} The argument is an easy modification of the above.
\end{proof}

The following improves upon a result in \cite[Prop.\ 6.3]{BMMS} where it is shown that
for a cubic $X\in\kc_8,$ so containing a plane, there exist at most finitely
many (up to isomorphism) cubics $X_1,\ldots,X_n\in\kc_8$ with $\ka_X\cong\ka_{X_1}\cong\ldots\cong\ka_{X_n}$.

\begin{cor}\label{cor:finiteFM}
For any given smooth cubic $X\subset\PP^5$ there exist up to isomorphism only finitely many smooth cubics $X'\subset\PP^5$ 
admitting a FM-equivalence $\ka_X\congpf\ka_{X'}$.
\end{cor}

\begin{proof}
The proof follows the argument for the analogous statement for K3 surfaces \cite{BrMac} closely,
but needs a modification at one point that shall be explained.

Due to the proposition, it suffices to prove that up to isomorphism there are
only finitely many cubics $X'$ such that there exists a Hodge isometry $\varphi\colon\widetilde H(\ka_X,\ZZ)\congpf \widetilde H(\ka_{X'},\ZZ)$.
Any such Hodge isometry induces a Hodge isometry $\varphi_T\colon T(\ka_X)\congpf T(\ka_{X'})$ and an isometry of lattices
$N(\ka_X)\congpf N(\ka_{X'})$. We may assume $$T(\ka_X)\subset A_2^\perp\subset \widetilde H(\ka_X,\ZZ)
~\text{ and }~A_2\subset N(\ka_X)$$ and similarly for $X'$. Note however that these inclusions need not be respected by $\varphi$.
The orthogonal complement of $T(\ka_X)^\perp\subset A_2^\perp$ is just $N(\ka_X)\cap A_2^\perp$ and the two inclusions
of $A_2^\perp$ induce two Hodge structures on $A_2^\perp$.
Note that if the Hodge isometry $\varphi_T$ can be extended to a Hodge isometry $A_2^\perp\congpf A_2^\perp$,
which can be interpreted as  a Hodge isometry $H^4(X,\ZZ)_{\rm prim}\cong H^4(X',\ZZ)_{\rm prim}$, then the Global Torelli
theorem \cite{Voisin} implies that $X\cong X'$. 

We first show that the set of isomorphism classes of lattices $\Gamma$ occuring as $N(\ka_{X'})\cap A_2^\perp$ is finite.
The required lattice theory is slightly more involved than the original one in \cite{BrMac}. Let us fix two
even lattices $\Lambda_1$  and $\Lambda$ (in our situation, $\Lambda_1=T(\ka_X)$ and
$\Lambda=A_2^\perp$). We show that up to isomorphisms there exist
only finitely many lattices $\Lambda_2$ occurring as the orthogonal complement of some
primitive embedding $\Lambda_1\hookrightarrow\Lambda$. For unimodular $\Lambda$ this is standard,
but the proof can be tweaked to cover the more general statement. Of course, it suffices to
show that  only finitely many discriminant forms $(A_{\Lambda_2},q_{\Lambda_2})$ can occur. Now
$G\coloneqq\Lambda/(\Lambda_1\oplus\Lambda_2)$ is naturally a finite subgroup of $\Lambda^*/(\Lambda_1\oplus\Lambda_2)$
of index $d=|{\rm disc}(\Lambda)|$. The first projection from $G\subset\Lambda^*/(\Lambda_1\oplus\Lambda_2)\subset A_{\Lambda_1}\oplus A_{\Lambda_2}$ defines an isomorphism of $G$ with a finite subgroup of $ A_{\Lambda_1}$. This leaves only finitely many possibilities
for the finite groups $G$ and $\Lambda^*/(\Lambda_1\oplus\Lambda_2)$.
Note that $\Lambda/(\Lambda_1\oplus\Lambda_2)\subset A_{\Lambda_1}\oplus A_{\Lambda_2}$
is isotropic but not necessarily the bigger $\Lambda^*/(\Lambda_1\oplus\Lambda_2)\subset A_{\Lambda_1}\oplus A_{\Lambda_2}$.
However, the restriction of the quadratic form to $\Lambda^*/(\Lambda_1\oplus\Lambda_2)$ takes values only in $(2/d^2)\ZZ/2\ZZ$.
For fixed $G\subset A_{\Lambda_1}$ the restriction of $q_{\Lambda_1}$ to $G$ can be extended
in at most finitely many ways to a quadratic form on $\Lambda^*/(\Lambda_1\oplus\Lambda_2)$ with values in $(2/d^2)\ZZ/2\ZZ$.
Now use the other projection $\Lambda^*/(\Lambda_1\oplus\Lambda_2)\twoheadrightarrow A_{\Lambda_2}$ to see that
there are only finitely many possibilities for the group $A_{\Lambda_2}$ and also for the quadratic form $q_{\Lambda_2}$.

To conclude the proof, we can assume that $\Gamma$ is fixed. For two Fourier--Mukai partners
realizing the fixed $\Gamma$,  any Hodge isometry $T(\ka_{X_1})\cong T(\ka_{X_2})$
can be extended to a Hodge isometry $T(\ka_{X_1})\oplus \Gamma\cong T(\ka_{X_2})\oplus\Gamma$.
As the finite index overlattices $ T(\ka_{X_i})\oplus \Gamma\subset H^4(X_i,\ZZ)_{\rm prim}$ are all
contained in $( T(\ka_{X_i})\oplus \Gamma)^*$, there are only finitely many choices for them, which allows one
to reduce to the case
that the Hodge isometry extends to a Hodge isometry $H^4(X_1,\ZZ)_{\rm prim}\cong H^4(X_2,\ZZ)_{\rm prim}$.
\end{proof}

Two  very general cubics have FM-equivalent K3 categories only if they are isomorphic:

\begin{cor}\label{cor:verygeneralnoFM}
Let $X$ be a smooth cubic with $\rk\, H^{2,2}(X,\ZZ)=1$. For a smooth cubic $X'$
there exists a FM-equivalence $\ka_X\cong\ka_{X'}$ if and only if $X\cong X'$.
\end{cor}

\begin{proof}
The assumption implies that $N(\ka_X)\cong A_2$. As any FM-equivalence $\ka_X\cong\ka_{X'}$
induces a Hodge isometry $\widetilde H(\ka_X,\ZZ)\cong\widetilde H(\ka_{X'},\ZZ)$, also
$N(\ka_{X'})\cong A_2$. Moreover, the natural inclusions of the transcendental 
lattices $T(\ka_X)\subset A_2^\perp$ and $T(\ka_{X'})\subset A_2^\perp$ are in fact equalities and
the induced Hodge isometry $T(\ka_X)\cong T(\ka_{X'})$ can therefore be read as a Hodge isometry
$H^4(X,\ZZ)_{\rm prim}\cong H^4(X',\ZZ)_{\rm prim}$, which by the Global Torelli theorem \cite{Voisin} implies that $X\cong X'$. 
\end{proof}

Note that in contrast,  very general projective K3 surfaces $S$, i.e.\ such that $\rho(S)=1$, usually have non-isomorphic FM-partners, see \cite{Og,St}.
The result may also be compared to the main result of \cite{BMMS} showing that for all cubic threefolds $Y\subset\PP^4$
the full subcategory $\langle\ko,\ko(1)\rangle^\perp\subset\Db(Y)$ determines $Y$.

%\begin{cor} Assume $X,X'\subset\PP^5$ are two smooth cubics such that there  exists
%a FM-equivalence $\ka_X\cong\ka_{X'}$. If $\rk\, H^{2,2}(X,\ZZ)>13$, then
%$X\cong X'$.
%\end{cor}
%
%\begin{proof}
%This follows more or less the standard argument. Any FM-equivalence $\ka_X\cong\ka_{X'}$ induces a Hodge isometry $\widetilde H(\ka_X,\ZZ)\cong\widetilde H(\ka_{X'},\ZZ)$
%and hence a Hodge isometry $T(\ka_X)\cong T(\ka_{X'})$. The latter extends to a Hodge isometry
%$T(\ka_X)\oplus\ZZ(-3)\cong T(\ka_{X'})\oplus\ZZ(-3)$ by ${\rm id}$ on $\ZZ(-3)$ which is declared to be of type $(1,1)$.
%Now use that there exists a primitive embedding of Hodge structures $T(\ka_X)\oplus\ZZ(-3)\,\hookrightarrow H^4(X,\ZZ)(-1)$,
%that sends $\ZZ(-3)$ to the line spanned by ${\rm c}_1(\ko(1))^2$,  and similarly for $X'$. By Nikulin's \cite[Thm.\ 1.14.4]{NikulinInt}
%the underlying lattice embedding is unique under our assumption, which leads to a
%Hodge isometry $H^4(X,\ZZ)\cong H^4(X',\ZZ)$ respecting the squares of the hyperplane classes. Hence, by the Global Torelli
%theorem $X\cong X'$.
%\end{proof}

\begin{remark}
In principle it should be possible to count FM-partners of $\ka_X$ for very general special cubics $X\in\kc_d$ (i.e.\  $\rk\, H^{2,2}(X,\ZZ)=2$).
On the level of Hodge theory, this amounts to counting the number of Hodge structures on $\widetilde\Lambda$ parametrized by $D$  which
are Hodge isometric to $\widetilde H(\ka_X,\ZZ)$ up to those that are Hodge isometric on $A_2^\perp$. The arguments should
follow \cite[Thm.\ 1.4]{HO}, see also \cite{St}, with the additional problem that $A_2^\perp$ is not unimodular.
%However, as the period map for cubics is not surjective, this would only provide an upper bound for the number
%of isomorphism classes of cubics $X'$ for which there exists a Hodge isometry $\widetilde H(\ka_{X'},\ZZ)\cong\widetilde H(\ka_X,\ZZ)$. 
%Note that due to Theorem \ref{thm:HodgeA}  this would yield control over the number of isomorphism classes
%of $X'$ with $\ka_{X'}\cong\ka_X$, at least for very general $X\in\kc_d$.
%
%In \cite{BMMS} it is shown that general cubics in $\kc_8$ with equivalent K3 categories are isomorphic. This should no longer hold
%for special ones in $\kc_8$, see \cite[Rem.\ 6.4]{BMMS}.
\end{remark}

As an immediate consequence of Lemma \ref{lem:dics11} we also note

\begin{cor}
Let $X$ be a special  cubic defining a very general point in $\kc_d$. Then
$$\rk( \widetilde H^{1,1}(\ka_X,\ZZ))=3\text{ and }{\rm disc}(\widetilde H^{1,1}(\ka_X,\ZZ))=d.$$
\end{cor}

\begin{remark}
Suppose $d$ satisfies ($\ast$$\ast'$) and is written as $d=k^2d_0$. Then $d_0$ also satisfies ($\ast$$\ast'$). The most interesting case
is when in fact $d_0$ satisfies ($\ast$$\ast$). Then for very general $X\in \kc_d$, there exists a twisted K3 surface $(S,\alpha)$ with $\alpha$
of order $k$ and such that $\ka_X\cong\Db(S,\alpha)$, see Lemma \ref{lem:OrderBrauer}. Moreover, there also exists a cubic $X'\in\kc_{d_0}$ such that
$\ka_{X'}\cong\Db(S)$. So, a K3 surface $S$ of the proper degree, with its various Brauer classes, is often related to more
than one smooth cubic $X$.
\end{remark}
%%%%%%%%%%%%%%%%%%%%%%%%%%%%%%%%%%%%%%%%%%
\subsection{} We are  interested in the group $\Aut(\ka_X)$
of isomorphism classes of FM-equivalences $\Phi\colon\ka_X\congpf\ka_X$. 
%We say that $\Phi\colon\ka_X\congpf\ka_X$ is of \emph{Fourier--Mukai type} if the composition
%$\Db(X)\to\ka_X\congpf\ka_X\to\Db(X)$ admits a Fourier--Mukai kernel $\kp\in\Db(X\times X)$.
As any FM-equivalence $\Phi$ induces a Hodge isometry
$$\Phi^H\colon\widetilde H(\ka_X,\ZZ)\congpf \widetilde H(\ka_X,\ZZ),$$
there is  a natural homomorphism
\begin{equation}\label{eqn:repAut}
\rho\colon\Aut(\ka_X)\to\Aut(\widetilde H(\ka_X,\ZZ)),~\Phi\mapsto\Phi^H.
\end{equation}
Here, $\Aut(\widetilde H(\ka_X,\ZZ))$ denotes the group of Hodge isometries. 
We say that $\Phi$ is \emph{symplectic} if the induced action on $\widetilde H^{2,0}(\ka_X)$, or equivalently
on $T(\ka_X)$, is the identity. The subgroup of symplectic autoequivalences shall be denoted
by $\Aut_{\rm s}(\ka_X)$. Thus, (\ref{eqn:repAut})  induces $$\rho\colon\Aut_{\rm s}(\ka_X)\to\Aut(\widetilde H^{1,1}(\ka_X,\ZZ)).$$

\begin{remark}
By $\Aut^+(\widetilde H(\ka_X,\ZZ))$ one denotes the subgroup of Hodge isometries preserving 
a given orientation of the four positive directions. We expect that
${\rm Im}(\rho)= \Aut^+(\widetilde H(\ka_X,\ZZ))$. This is known
if $\ka_X\cong\Db(S)$,  see \cite{HMS}, and one inclusion can be proved 
for non-special cubics, see Theorem \ref{thm:noFMvg}.
\end{remark}

\begin{ex}\label{exa:sphtw}
The most important autoequivalences of K3 categories, responsible for the complexity of
the groups $\Aut(\Db(S))$ and $\Aut(\ka_X)$, are spherical twists. Associated with any spherical
object $A\in\ka_X$, i.e.\ $\Ext^*(A,A)\cong H^*(S^2)$, there exists a FM-equivalence
$$T_A\colon\ka_X\congpf\ka_X$$ that sends $E\in\ka_X$ to the cone $T_A(E)$ of the evaluation map
$R{\rm Hom}(A,E)\otimes A\to E$. This is indeed a FM-equivalence -- its kernel
can be described as the cone of the  composition $A\!^\vee\boxtimes A\stackrel{tr}{\to}\ko_\Delta\to
({\rm id},i)^*(\ko_\Delta)$. Here, $({\rm id},i)^*$ is the left adjoint
$\Db(X\times X)\to\Db(X)\boxtimes\ka_X$ and $A\!^\vee\in\ka_X(-2)$ is the image of the classical
dual of $A$ in ${\rm D}^{\rm b}(X)$ under the left adjoint of $\ka_X(-2)\subset{\rm D}^{\rm b}(X)$. (With these choices the cone is
contained in $\ka_X(-2)\boxtimes\ka_X$ and would indeed induce a functor ${\rm D}^{\rm b}(X)\to\ka_X$ that
is trivial on ${} ^\perp\!\ka_X$.)

The action of the spherical twist $T_A\colon \ka_X\congpf\ka_X$ on $\widetilde H(\ka_X,\ZZ)$
is given by the reflection $s_\delta\colon v\mapsto v+\langle v,\delta\rangle\cdot \delta$, where
$\delta\in \widetilde H^{1,1}(\ka_X,\ZZ)$ is the Mukai vector of $A$.
\end{ex}

%\begin{ex}\label{exa:Psi}
In \cite[Sect.\ 4]{Kuz3} Kuznetsov considers the functor $$\Psi\colon\ka_X\to\ka_X,~~ÊE\mapsto i^*(i_*E\otimes\ko_X(1))[-1],$$
which turns out to be an equivalence satisfying $\Psi^3\cong[-1]$. Clearly, by construction $\Psi$ is a FM-equivalence.
In fact, for the proof that $\ka_X$ is a K3 category this functor is crucial.
%\end{ex}
Define $$\Phi_0\coloneqq \Psi[1],$$
which satisfies $\Phi_0^3\cong [2]$.

\begin{prop}\label{prop:GenericPhi0} The autoequivalence $\Phi_0\colon\ka_X\congpf\ka_X$ 
is symplectic and the induced action $\Phi_0^H\colon\widetilde H(\ka_X,\ZZ)\congpf \widetilde H(\ka_X,\ZZ)$
corresponds to the element in  $\OO(A_2)$ that is given by the cyclic
permutation of the  roots $\lambda_1,\lambda_2,-\lambda_1-\lambda_2$. 
\end{prop}

\begin{proof} 
As the action on cohomology is independent of the specific
cubic $X\subset\PP^5$,
we can assume that the transcendental lattice $T(\ka_X)\subset\widetilde H(\ka_X,\ZZ)$
is of odd rank. However,  $\pm{\rm id}$ are the only Hodge isometries
of an irreducible Hodge structure of weight two of K3 type of odd rank, cf.\ \cite[Cor.\ 3.3.5]{HK3}, and, as $\Phi_0^3\cong[2]$
acts trivially on $\widetilde H(\ka_X,\ZZ)$, we must have $\Phi_0^H={\rm id}$ on $T(\ka_X)$, i.e.\ $\Phi_0$ is symplectic.

If $X$ is a cubic with $A_2\cong\widetilde H^{1,1}(\ka_X,\ZZ)$, then $\Phi_0^H$ corresponds to
an element in $\OO(A_2)$. As $\Phi_0$ is symplectic, $\Phi_0^H={\rm id}$ on $A_2^\perp$ and hence $\Phi_0^H={\rm id}$
on the discriminant group $A_{A_2}$. Therefore, $\Phi_0^H\in{\mathfrak S}_3$, see Remark \ref{rem:OA_2}. 
For a cubic $X$ such that $\ka_X\cong \Db(S)$, we know that $\Phi_0^H$ must be orientation preserving 
by \cite{HMS} and thus $\Phi_0^H\in{\mathfrak A}_3\cong\ZZ/3\ZZ$ in general. 

It remains to show that $\Phi_0^H\ne{\rm id}$.
% In principle one could do this by computing $\Phi_0$ of the projection of $\ko_\ell(1)$ under
%$\Db(X)\to\ka_X$ which realizes $\lambda_1$. This would however require  computing
%explicitly the corresponding matrix factorization and  proving that  $\Phi_0$ does
%not preserve its cohomology class.
One way to see this relies on a direct computation. Another possibility is to use the recent result of Bayer and Brigeland
\cite{BB} confirming Bridgeland's conjecture in \cite{BrK3}  in the case of a K3 surface $S$  of Picard rank one. More precisely, due to \cite[Thm.\ 1.4]{BB}   for a K3 surface $S$ with $\rho(S)=1$ the subgroup of $\Aut(\Db(S))$ of autoequivalences acting trivially on $\widetilde H(S,\ZZ)$ is the
product of $\ZZ[2]$ and the free group generated by squares of spherical twists $T_E^2$ associated with
spherical vector bundles $E$ on $S$. (That this is a reformulation of Bridgeland's original conjecture
for $\rho(S)=1$ had also been observed  by Kawatani \cite{Kaw}.)
Hence, if $\Phi_0^H={\rm id}$, then $\Phi_0=({\Asterisk_i} T_{E_i}^2)\circ[2k]$,
but then clearly $\Phi_0^3$ could not be isomorphic to the double shift $[2]$.
\end{proof}

\begin{cor}\label{cor:Z3Z}
For every smooth cubic $X\subset \PP^5$ the group of symplectic  FM-autoequivalences
${\rm Aut}_s(\ka_X)$ contains an infinite cyclic group $\ZZ\subset\Aut_s(\ka_X)$ generated by $\Phi_0$
such that  $$\ZZ\cdot[2]\subset\ZZ$$ is a subgroup of index three and such that the natural map
$\rho\colon\Aut_s(\ka_X)\to\Aut(\widetilde H(\ka_X,\ZZ))$ defines an isomorphism of the quotient $\ZZ/\ZZ\cdot [2]$  with the subgroup
 ${\mathfrak A}_3\subset\OO(A_2)\subset \OO( \widetilde H(\ka_X,\ZZ))$ of alternating permutations of the roots
 $\lambda_1,\lambda_2,-\lambda_1-\lambda_2$ of $A_2$.\qed
\end{cor}

\begin{remark}\label{rem:univcoveract}
The subgroup ${\rm SO}(A_2)\subset{\rm O}(A_2)$ of orientation preserving
isometries of $A_2$ is ${\mathfrak A}_3\times\ZZ/2\ZZ\cong\ZZ/3\ZZ\times\ZZ/2\ZZ$, see Remark
\ref{rem:OA_2}. Its action can be `lifted' to an action on $\ka_X$ via the natural extension
$$0\to\ZZ\cdot[2]\to(\ZZ\cdot\Phi_0\times\ZZ\cdot[1])/(\Phi_0^3-[2])\to{\rm SO}(A_2)\to 0,$$ which can be seen
as induced by the universal cover of ${\rm SO}(A_2\otimes\RR)$. Clearly, the group in the middle is still infinite cyclic.
\end{remark}

Inspired  by Bridgeland's conjecture for K3 surfaces in \cite{BrK3}, we state the following (see \cite[Sect.\ 5.4]{HuyStab} explaining this reformulation):

\begin{conj}\label{conj:Brid} \emph{There exists an isomorphism
$${\rm Aut}_{\rm s}(\ka_X)\cong\pi_1^{\rm st}[P_0/{\rm O}].$$}
\end{conj}

Here, $P\subset \PP(\widetilde H^{1,1}(\ka_X,\ZZ)\otimes\CC)$ is the period domain defined analogously to $D$ and $Q$ in 
Section \ref{sec:defperiod} and
$P_0\coloneqq P\setminus\bigcup\delta^\perp$, with the union over all $(-2)$-classes $\delta\in\widetilde H^{1,1}(\ka_X,\ZZ)$.
Moreover, ${\rm O}\subset{\rm O}(\widetilde H^{1,1}(\ka_X,\ZZ))$ is the subgroup of all isometries acting trivially on the discriminant.
However, contrary to the case of untwisted K3 surfaces we do not even have a natural map between these two groups at the moment.
%%%%%%%%%%%%%%%%%%%%%%%%%%%%%%%%%%%%%%
\subsection{} The cubic K3 category $\ka_X$ can also be described as a category
of graded matrix  factorizations, see \cite{Orlov}. More precisely, there exists an exact linear equivalence
$$\ka_X\cong {\rm MF}(W,\ZZ).$$
Here, $W\in R\coloneqq k[x_0,\ldots,x_5]$ is a cubic polynomial defining $X$. The
objects of ${\rm MF}(W,\ZZ)$ are pairs $(K\stackrel{\alpha}{\to}L,L\stackrel{\beta}{\to} K(3))$,
where $K$ and $L$ are finitely generated, free, graded $R$-modules and $\alpha,\beta$ are
graded $R$-module homomorphisms with $\beta\circ\alpha=W\cdot{\rm id}=\alpha\circ\beta$. 
Recall that $K(n)$ for a graded $R$-module $K=\bigoplus K_i$ is the graded module
with $K(n)_i=K_{n+i}$. Homomorphisms in ${\rm MF}(W,\ZZ)$ are the obvious ones modulo
those that are homotopic to zero (everything $\ZZ/2\ZZ$-periodic). 

The \emph{shift functor}  that makes ${\rm MF}(W,\ZZ)$ a triangulated category
is given by $$(K\stackrel{\alpha}{\to}L,L\stackrel{\beta}{\to} K(3))[1]=(L\stackrel{-\beta}{\to}K(3),K(3)\stackrel{-\alpha}{\to} L(3)).$$
%Thus, the double shift is $$(K\stackrel{\alpha}{\to}L,L\stackrel{\beta}{\to} K(3))[2]\cong
%(K(3)\stackrel{\alpha}{\to}L(3),L(3)\stackrel{\beta}{\to} K(6)).$$

Viewing $\ka_X$ as the category of graded matrix factorizations allows one to describe $\Phi_0$ in Proposition \ref{prop:GenericPhi0} alternatively as follows.
Consider the \emph{grade shift functor}
\begin{eqnarray*}
\Phi_0\colon {\rm MF}(W,\ZZ)&\congpf& {\rm MF}(W,\ZZ)\\
 (K\stackrel{\alpha}{\to}L,L\stackrel{\beta}{\to} K(3))&\mapsto&(K(1)\stackrel{\alpha}{\to}L(1),L(1)\stackrel{\beta}{\to} K(4)).
 \end{eqnarray*}
Then, obviously, $$\Phi_0^3\cong [2].$$ 
%Using the equivalence $\ka_X\cong {\rm MF}(W,\ZZ)$, we consider $\Phi_0$ as an auto-equivalence
%of $\ka_X$.

Note that $\Phi_0$ constructed in this way coincides with the FM-equivalence of Proposition \ref{prop:GenericPhi0}, 
see \cite[Prop. 5.8]{BFK}.
%one would need to
%carry out a computation as in \cite[Sec.\ 5]{BFK} or \cite[Sec.\ 5]{KMvdB}. However, for abstract reasons
%$\Phi_0$ with this description is certainly also  of Fourier--Mukai type, i.e.\ $\Phi_0\in\Aut(\ka_X)$. To see this, note first
%that $\Phi_0$ obviously has an enhancement, i.e.\ it lifts to the dg-enhancement of ${\rm MF}(W,\ZZ)$ provided
%by the category ${\rm MF}^{\rm dg}(W,\ZZ)$ (which has the same objects and its morphisms are
%$\ZZ/2\ZZ$-periodic morphisms of complexes of arbitrary degree). As $i_*$ and $i^*$ are of Fourier--Mukai type and thus admit %dg-enhancements, also the composition $i_*\circ\Phi_0\circ i^*$ does.  By \cite{Toen} this implies that $i_*\circ\Phi\circ i^*$ is a %Fourier--Mukai functor and
%hence $\Phi_0$ is of Fourier--Mukai type. 

%%%%%%%%%%%%%%%%%%%%%%%%%%%%%%%%%%%%

\section{The Fano variety}\label{sec:Fano}

For the sake of completeness, let us also mention the recent results of Addington \cite{Add} building upon an observation
of Hassett \cite{HassettComp}, see also \cite{MacStel}.  For this consider the Fano variety of lines $F(X)$,
which, due to work of Beauville and Donagi \cite{BD}, is a four-dimensional irreducible holomorphic symplectic variety deformation equivalent to ${\rm Hilb}^2({\rm K3})$.
\smallskip

\noindent
$\bullet$ For a smooth cubic $X$ and its period $x\in D$ the following two conditions are equivalent:\\
i) $x\in D_d$ such that $d$ satisfies  ($\ast$$\ast$);\\
ii) $F(X)$ is birational to a moduli space of stable sheaves $M(v)$ on some K3 surface $S$.

\smallskip

\noindent
$\bullet$ For a smooth cubic $X$ and its period $x\in D$ the following two conditions are equivalent:\\
iii) $x\in D_d$ such that  there exist integers $n$ and $a$ with $da^2=2(n^2+n+1)$;\\
iv) $F(X)$ is birational to the Hilbert scheme
${\rm Hilb}^2(S)$ of some K3 surface $S$.

Obviously, iv) implies ii) or, equivalently and after a moment's thought, iii) implies i). See \cite{GS} for a discussion
of the relation between rationality of the cubic $X$ and condition iii) (or, equivalently, iv)).

\begin{prop}\label{prop:analAdd}
For the period $x$ of a smooth cubic $X$ the following two conditions are equi\-valent:\\
{\rm i)} $x\in D_d$ with $d$ satisfying ($\ast$$\ast'$);\\
{\rm ii)} $F(X)$ is birational to a moduli space of stable twisted sheaves on some K3 surface.
\end{prop}

\begin{proof}
The argument is an adaptation of Addington's proof
\cite{Add}. Note however that in the twisted case the transcendental lattice cannot play the same role as in the untwisted
case. This was observed in \cite{HuSt}, where it was shown that twisted K3 surfaces $(S,\alpha)$, $(S',\alpha')$ with Hodge isometric transcendental
lattices, $T(S,\alpha)\cong T(S',\alpha')$, need not be derived equivalent.

Following Markman \cite{Mark} for every hyperk\"ahler manifold $Y$ deformation equivalent to $
{\rm Hilb}^2(S)$ of a K3 surface $S$ there exists a distinguished primitive embedding
$H^2(Y,\ZZ)\subset\widetilde\Lambda$ orthogonal to a vector $v\in\widetilde\Lambda$ with $(v.v)=2$. The Hodge structure of $H^2(Y,\ZZ)$ extends
to a Hodge structure on $\widetilde\Lambda$ such that $v$ is of type $(1,1)$. Moreover, $Y$ and $Y'$ are birational if and only if
there exists a Hodge isometry $H^2(Y,\ZZ)\cong H^2(Y',\ZZ)$ that extends to a Hodge isometry $\widetilde\Lambda\cong\widetilde\Lambda$. 
For a moduli space $M(v)$ of $\alpha$-twisted stable sheaves on a K3 surface $S$ with 
primitive $v\in \widetilde H^{1,1}(S,\alpha,\ZZ)$ such that $(v.v)=2$ the universal family
induces the distinguished embedding (see \cite[Thm.\ 3.19]{Yosh})
$$H^2(M(v),\ZZ)\cong v^\perp\,\hookrightarrow \widetilde H(S,\alpha,\ZZ).$$
Similarly, and this is the 
other crucial input, Addington shows in \cite[Cor.\ 8]{Add} that for the Fano variety of lines the universal family of lines
induces this distinguished embedding
$$H^2(F(X),\ZZ)\cong\lambda_1^\perp\,\hookrightarrow \widetilde H(\ka_X,\ZZ)\cong\widetilde\Lambda.$$
Hence, $F(X)$ and $M(v)$ are birational if and only if there exists a Hodge isometry
\begin{equation}\label{eqn:Ftw}
\widetilde H(\ka_X,\ZZ)\cong\widetilde H(S,\alpha,\ZZ)
\end{equation}
for some twisted K3 surface $(S,\alpha)$ that restricts to $H^2(F(X),\ZZ)\cong H^2(M(v),\ZZ)$.
Due to Proposition \ref{prop:DtwDd}, the existence of a Hodge isometry (\ref{eqn:Ftw}) is equivalent to $x\in D_d$ with $d$ satisfying ($\ast$$\ast'$).
This proves that ii) implies i).

Conversely, for a Hodge isometry (\ref{eqn:Ftw}) consider a primitive vector $v\in \widetilde H^{1,1}(S,\alpha,\ZZ)$ 
(the image of $\lambda_1$) in
the orthogonal complement of $H^2(F(X),\ZZ)\,\hookrightarrow\widetilde H(\ka_X,\ZZ)\cong \widetilde H^{2}(S,\alpha,\ZZ)$ and the
induced moduli space $M(v)$ of stable $\alpha$-twisted sheaves. Write $v=(r,\ell,s)$. If $r\ne0$, then for $v$ or $-v$ the moduli space $M(v)$
is indeed non-empty. For $r=0$ observe that $(v)^2>0$ and hence $(\ell)^2>0$. Again by passing to $-v$ if necessary, one can assume that
$(\ell.H)>0$ for the polarization $H$. That the moduli space is non-empty in this case was shown in \cite[Cor.\ 3.5]{Yosh2}. (Note that
for $r=0$ twisted sheaves can also be considered as untwisted ones.)
In \cite{Add} the case $r=0$ is dealt with by a reflection associated with $\ko$, which
does not work in the twisted situation. 

To conclude, compose the Hodge isometry $H^2(F(X),\ZZ)\cong v^\perp$, given by the choice of $v$, with $H^2(M(v),\ZZ)\cong v^\perp$,
induced by the universal family as above. By construction, it extends to a Hodge isometry $\widetilde\Lambda\cong\widetilde\Lambda$ and, therefore,
$F(X)$ and $M(v)$ are birational.
\end{proof}

%%%%%%%%%%%%%%%%%%%%%%%%
\section{Deformation theory}\label{sec:defo}
This section contains  two results on the deformation theory of equivalences $\Db(S,\alpha)\cong\ka_X$ resp.\
$\ka_{X'}\cong\ka_X$ that are crucial for the main results of the paper. The
techniques have been developed by Toda \cite{Toda}, Huybrechts--Macr\`i--Stellari \cite{HMS},
Huybrechts--Thomas \cite{HT}, and in the present setting by Addington--Thomas \cite{AT}. 
We follow \cite{AT} quite closely and often only indicate the additional difficulties and how to deal with them.
%%%%%%%%%%%%%
\subsection{} We first consider FM-equivalences $\ka_{X'}\cong\ka_X$ between the K3 categories of two cubics $X$ and $X'$
and study under which condition they deform sideways with $X$ and $X'$.

\begin{thm}\label{thm:Defo1}
Consider two families of smooth cubics $\kx,\kx'\to T$ over a smooth base $T$  and with distinguished fibres
$X\coloneqq \kx_0$ and $X'\coloneqq \kx'_0$, respectively. Assume $\Phi\colon\ka_{X'}\congpf\ka_X$ is a FM-equivalence
inducing a Hodge isometry $\varphi\colon\widetilde H(\ka_{X'},\ZZ)\congpf\widetilde H(\ka_X,\ZZ)$ that remains
a Hodge isometry $\varphi_t\colon\widetilde H(\ka_{\kx_t'},\ZZ)\congpf \widetilde H(\ka_{\kx_t},\ZZ)$ under parallel
transport  for all $t\in T$. 

Then $\Phi$ deforms sideways to FM-equivalences $\Phi_t\colon\ka_{\kx_t'}\congpf\ka_{\kx_t}$
for all $t$ in a Zariski open neighbourhood $0\in U\subset T$.
\end{thm}

\begin{proof}
The argument is a variant of the deformation theory in \cite{AT}. We only indicate the necessary modifications.

As by assumption $\Phi$ is a FM-equivalence, the composition $$\xymatrix{\Phi_P\colon\Db(X')\ar[r]&\ka_{X'}\ar[r]^-\sim_\Phi&\ka_X
\ar@{^(->}[r]&\Db(X)}$$
is a FM-functor with some kernel $P\in\Db(X'\times X)$ contained in $\ka_{X'}(-2)\boxtimes\ka_X$. It suffices to show that $P$ deforms to $P_t\in\Db(\kx_t'\times\kx_t)$
for $t$ in some open neighbourhood $0\in U\subset T$, because the conditions for $\Phi_{P_t}$ to factorize
via a  functor $\Phi_t\colon\ka_{\kx_t'}\to\ka_{\kx_t}$ and for this functor $\Phi_t$ to define an equivalence are both Zariski open.
Indeed, $\Phi_t$ takes values in $\ka_{\kx_t}$ if and only if its composition with the projection 
$\Db(\kx_t)\to ^\perp\!\ka_{\kx_t}=\langle\ko_{\kx_t},\ko_{\kx_t}(1),\ko_{\kx_t}(2)\rangle$ is trivial. The composition, however, is again of FM-type and the vanishing of a FM-kernel is a Zariski open condition. Similarly, whether $\Phi_t$ is an equivalence can be detected by composing it with
its adjoints and then checking whether the natural map to the kernel of the identity is an isomorphism, again a Zariski open condition.

The crucial part is to understand the first order deformations, the higher order obstructions are dealt with by the $T^1$-lifting
property, see \cite[Sec.\ 7.2]{AT} and \cite[Sec.\ 3.2]{HMS}.
First note that by results of Kuznetsov \cite{Kuz2} one has 
$$\HoH^*(\ka_{X'})\cong\Ext^*_{X'\times X}(P,P)\cong\HoH^*(\ka_X).$$
This allows one to compare the first order deformations $$\kappa_{X'}\in H^1(T_{X'})\subset\HoH^2(X')
\text{ and } \kappa_X\in H^1(T_X)\subset\HoH^2(X)$$
corresponding to some tangent vector $v\in T_0T$ of $T$ at $0$. Due to a result of Toda \cite{Toda} (cf.\ \cite[Thm.\ 7.1]{AT}) it suffices to show that
under $\HoH^2(X')\to\Ext^2_{X'\times X}(P,P)$ resp.\ $\HoH^2(X)\to\Ext^2_{X'\times X}(P,P)$ the classes
$\kappa_{X'}$ and $\kappa_X$ are mapped to the same class.
For this consider the following diagram (cf.\ \cite[Prop.\ 6.2]{AT})
$$\xymatrix{\HoH_2(X')\ar @{} [dr] |{{\rm (1)}}\ar[d]_{\kappa_{X'}}\ar[r]^-\sim&\HoH_2(\ka_{X'})\ar @{} [dr] |{{\rm (2)}}\ar[d]^{\alpha}\ar[r]_{\Phi^{\HoH_*}_P}^-\sim&\HoH_2(\ka_X)\ar[d]^{\bar\kappa_X}\ar[r]^-\sim\ar @{} [dr] |{{\rm (3)}}&\HoH_2(X)\ar[d]^{\kappa_X}\\
\HoH_0(X')\ar @{} [drrr] |{{\rm (4)}}\ar[d]\ar@{->>}[r]&\HoH_0(\ka_{X'})\ar[r]_{\Phi^{\HoH_*}_P}^-\sim&\HoH_0(\ka_X)\ar@{^(->}[r]&\HoH_0(X)\ar[d]\\
H^*(X')\ar[dr]\ar[rrr]_{\Phi_P^{H}}&&&H^*(X)\\
&\widetilde H^*(\ka_{X'})\ar[r]^-\sim_\varphi&\widetilde H^*(\ka_X)\ar[ur]&}$$
By $ H^*(X)\cong\HoH_*(X)$ we denote the HKR-isomorphism (see \cite{Cald}) post-composed with $\sqrt{\rm td}\wedge(~)$ and, so in particular, $\HoH_2(X)\cong H^1(\Omega_X^3)$ with chosen generator $\sigma_X$. Similarly for $X'$, where we choose the generator
$\sigma_{X'}\in H^1(\Omega_{X'}^3)\cong\HoH_2(X')$ such that its image yields $\sigma_X$.
Furthermore, $\bar\kappa_X$ denotes the image of $\kappa_X$ under the projection $\HoH^2(X)\to\HoH^2(\ka_X)$, see \cite{Kuz2},
and $\alpha\coloneqq \Phi^{\HoH^*}(\bar\kappa_X)$.

We aim at showing that (1) is commutative. For this note first that  (4) is induced by
the FM-transform $\Phi_P\colon\Db(X')\to\Db(X)$ and hence commutative due to \cite{MS}. The commutativity of (2) is obvious,
as Hochschild (co)homology is respected by equivalences, and commutativity of (3) is the analogue of \cite[Prop.\ 6.1]{AT}.
(Recall that $\Phi_P$ does not necessarily induce a map $\Phi_P^{\HoH^*}$, as it is not fully faithful.)

The first order version of the  assumption on the Hodge isometry $\varphi$ is the statement that the diagram
$$\xymatrix{%&T_{S,0}\ar[dr]\ar[dl]&\\
H^1(T_{X'})\ar@{^(->}[d]^{\sigma_{X'}}&T_{0}T\ar[r]\ar[l]&H^1(T_X)\ar@{^(->}[d]^{\sigma_X}\\
H^{2,2}(X')_{\rm prim}\ar@{^(->}[d]&&H^{2,2}(X)_{\rm prim}\ar@{^(->}[d]\\
\widetilde H^{1,1}(\ka_{X'})\ar[rr]_\varphi^-\sim&&\widetilde H^{1,1}(\ka_{X})
}$$
is commutative.
%(Note that $\varphi$ might not map $H^{2,2}(X')$ into $H^{2,2}(X)$, i.e.\ it is
%in general different from $\Phi_P^{H}$.)
Using the ring-module isomorphism
$(\HoH^*,\HoH_*)\cong(H^*(\bigwedge^*T),H^*(\Omega^*))$ for $X'$, this implies
that the image in $H^*(X')$ of $\sigma_{X'}\in \HoH_2(X')$ under contraction with $\kappa_{X'}$ 
is mapped to the image of $\sigma_{X}$ under contraction with $\kappa_X$. As $\HoH_2(X')$ is one-dimensional,
this shows that also  (1) is commutative.

Therefore,  in  the diagram
$$\xymatrix{\HoH^2(X')\ar[r]\ar[d]_{\sigma_{X'}}&\HoH^2(\ka_{X'})\ar@{^(->}[d]^{\sigma_{\ka_{X'}}}\ar[r]^-\sim\ar @{} [dr] |{\circlearrowleft}&\HoH^2(\ka_X)
\ar@{^(->}[d]^{\sigma_{\ka_{X}}}\\
\HoH_0(X')\ar[r]&\HoH_0(\ka_{X'})\ar[r]^-\sim&\HoH_0(\ka_X)}$$
the image of $\kappa_{X'}\in\HoH^2(X')$ under the two compositions $\HoH^2(X')\to\HoH_0(\ka_{X'})\cong \HoH_0(\ka_X)$ coincide.
As the contraction $\HoH^2(\ka)\,\hookrightarrow \HoH_0(\ka)$ is injective (as for K3 surfaces), this implies
that the image of $\kappa_{X'}$ under $\HoH^2(X')\to \HoH^2(\ka_X)$ is indeed $\bar\kappa_X$ as claimed.

As in \cite{AT}, the deformation of $P$ to first and then, by $T^1$-lifting property, to higher order is unique,
for $\Ext^1_{X'\times X}(P,P)\cong\HoH^1(\ka_{X})=0$ by \cite{Kuz2}.
\end{proof}

%%%%%%%%%%%%%%%%%%%
\subsection{} We now come to the more involved situation of equivalences $\Db(S,\alpha)\cong\ka_X$ and their deformations.

\begin{thm}\label{thm:Defo2}
Consider two families $\kx,\ks\to T$ of smooth cubics and K3 surfaces, respectively, over a smooth base $T$.
Denote the distinguished fibres by
$X\coloneqq \kx_0$, $S\coloneqq \ks_0$ and let $\alpha_t\in{\rm Br}(\ks_t)$ be a deformation
of a Brauer class $\alpha\coloneqq\alpha_0$ on $S$.
Assume $\Phi\colon\Db(S,\alpha)\congpf\ka_X$ is an equivalence
inducing a Hodge isometry $\varphi\colon\widetilde H(S,\alpha,\ZZ)\congpf\widetilde H(\ka_X,\ZZ)$ that remains
a Hodge isometry $\varphi_t\colon\widetilde H(\ks_t,\alpha_t,\ZZ)\congpf \widetilde H(\ka_{\kx_t},\ZZ)$ under parallel
transport  for all $t\in T$. 

Then $\Phi$ deforms sideways to equivalences $\Phi_t\colon\Db(\ks_t,\alpha_t)\congpf\ka_{\kx_t}$
for all $t$ in a Zariski open neighbourhood $0\in U\subset T$.
\end{thm} 

\begin{proof} Let us fix representatives $\alpha_t=\{\alpha_{t,ijk}\}$ for the Brauer classes on $\ks_t$
 and  a family $E_t$ of locally free $\{\alpha_{t,ijk}\}$-twisted sheaves on the fibres $\ks_t$ in a Zariski
 open  neighbourhood of $0\in U\subset T$.

The proof now consists of copying  \cite[Sec.\ 6, 7]{AT}. However, the techniques have to be adapted to the
twisted case, which sometimes causes additional problems as certain fundamental issues
related to Hochschild (co)homology have not been addressed in the twisted setting.
For certain parts we choose ad hoc arguments to reduce to the untwisted case, for others we rely
on Reinecke  \cite{Reinecke}. 

Section 6 in \cite{AT} deals with Hochschild (co)homology. For a twisted variety
$(Z,\alpha)$ one defines
$\HoH^n(Z,\alpha)\coloneqq \Ext^n_{(Z,\alpha^{-1})\times(Z,\alpha)}(\ko_\Delta,\ko_\Delta)$. Here,
$(Z,\alpha^{-1})\times(Z,\alpha)$ denotes the twisted variety $(Z\times Z,\alpha^{-1}\boxtimes \alpha)$.
Note that $\ko_\Delta$ is indeed an $(\alpha^{-1}\boxtimes\alpha)$-twisted sheaf.
Similarly, one defines $\HoH_n(Z,\alpha)\coloneqq \Ext^{d-n}_{(Z,\alpha^{-1})\times(Z,\alpha)}(\Delta_*\omega_Z^{-1},\ko_\Delta)$, where $d=\dim(Z)$.
Composition makes $\HoH_*(Z,\alpha)$ a right $\HoH^*(Z,\alpha)$-module. Moreover, there are natural isomorphisms
\begin{eqnarray*}
\HoH^n(Z,\alpha)= \Ext^n_{(Z,\alpha^{-1})\times(Z,\alpha)}(\ko_\Delta,\ko_\Delta)&\cong&\Ext^n_Z(\Delta^*\ko_\Delta,\ko_Z)\\
&\cong&\Ext^n_{Z\times Z}(\ko_\Delta,\ko_\Delta)=\HoH^n(Z)
\end{eqnarray*}
and
\begin{eqnarray*}\HoH_n(Z,\alpha)= \Ext^{d-n}_{(Z,\alpha^{-1})\times(Z,\alpha)}(\Delta_*\omega_Z^{-1},\ko_\Delta)
&\cong& \Ext^{d-n}_Z(\ko_Z,\Delta^*\ko_\Delta)\\
& \cong&\Ext^{d-n}_{Z\times Z}(\Delta_*\omega_Z^{-1},\ko_\Delta)
=\HoH_n(Z).
\end{eqnarray*}
In particular, the HKR-isomorphisms post-composed with ${\rm td}^{-1/2}\llrcorner(~)$
resp.\ ${\rm td}^{1/2}\wedge(~)$
yield isomorphisms $$I\colon\HoH^n(Z,\alpha)\congpf\bigoplus_{i+j=n} H^i(\Lambda^j T_Z)\text{ and } I\colon\HoH_n(Z,\alpha)\congpf\bigoplus_{j-i=n}H^i(\Omega_Z^j).$$
Note that these isomorphisms are again compatible with the ring and module structures on both sides, which follows from the fact that the
isomorphisms $\HoH^*(Z,\alpha)\cong\HoH^*(Z)$ and $\HoH_*(Z,\alpha)\cong \HoH_*(Z)$ are. The latter  is a consequence
of the functoriality properties
of $\Delta_!$, $\Delta_*$ and $\Delta^*$. 

For  a twisted K3 surface $(S,\alpha)$ one has $\HoH_2(S,\alpha)\cong H^0(\omega_S)=\CC\cdot\sigma_S$ and
the following diagram commutes
$$\xymatrix{\HoH^2(S,\alpha)\ar[d]^{\cdot\sigma_S}\ar[r]_-I^-\sim& H^0(\bigwedge^2T_S)\oplus H^1(T_S)\ar[d]^{\llrcorner\sigma_S}\oplus H^2(\ko_S)\\\HoH_0(S,\alpha)\ar[r]_-I^-\sim&
H^{0,0}\oplus H^{1,1}(S)\oplus H^{2,2}(S).}$$

Let us  now consider the  fully faithful functor $\Phi_P\colon\Db(S,\alpha)\congpf\ka_X\,\hookrightarrow\Db(X)$ between the twisted K3 surface $(S,\alpha)$ and the smooth cubic $X$,
where $P\in\Db((S,\alpha^{-1})\times X)$.
Then as in \cite[Sec.\ 6.1]{AT} one obtains  natural maps
$$\Phi_P^{\HoH^*}\colon\HoH^*(X)\to\HoH^*(S,\alpha)\text{ and }\Phi_P^{\HoH_*}\colon \HoH_*(S,\alpha)\to \HoH_*(X)$$
compatible with the module structures, i.e.\ $\Phi_P^{\HoH_*}(a)\circ c=\Phi_P^{\HoH_*}(a\circ\Phi_P^{\HoH^*}(c))$
for all $a\in\HoH_*(S,\alpha)$ and $c\in\HoH^*(X)$.
This has been checked by Reinecke in \cite[Sec.\ 4]{Reinecke}.

The remaining input in the proof of \cite[Prop.\ 6.2]{AT} is the commutativity of the untwisted version of the following diagram:
\begin{equation}\label{eqn:remcom}
\xymatrix{\HoH_0(S,\alpha)\ar[d]_{I^B}^-\wr\ar[rr]^{\Phi_P^{\HoH_*}}&&\HoH_0(X)\ar[d]^I_-\wr\\
\bigoplus H^{p,p}(S)\ar[rr]^{\Phi_P^{H,B}}&&\bigoplus H^{p,p}(X)}
\end{equation}
Note that defining the induced action on cohomology requires the lift of $\alpha$ to a class $B\in H^2(S,\QQ)$, see \cite{HuyInt,HuSt}. Moreover, the usual
HKR isomorphism $I$ post-composed with ${\rm td}^{1/2}\wedge(~)$ needs to be twisted further to $I^B\coloneqq\exp(B)\circ I$.
%\footnote{Assume $\alpha$ is of order $r$ and choose a lift $B=(1/r)B_0$ with $B_0\in H^2(S,\ZZ)$ of it.}

In principle, one could try to prove the commutativity of (\ref{eqn:remcom}) 
by rewriting the existing untwisted theory, in particular \cite{Cald,MS}, for the twisted situation. Instead, we
follow Yoshioka \cite{Yosh} and reduce  everything  to the untwisted case
by pulling back to a Brauer--Severi variety. We briefly review his approach and explain
how to apply it to our situation.

Following \cite{Yosh} we pick a locally free $\alpha=\{\alpha_{ijk}\}$-twisted
sheaf $E=\{E_i,\varphi_{ij}\}$ on a twisted variety $(Z,\alpha)$ and associate to it the projective bundle $\pi\colon Y\coloneqq\PP(E)\to Z$,
which naturally comes with a $\pi^*\alpha^{-1}$-twisted line bundle $L\coloneqq\ko_\pi(1)$. The pull-back of any $\alpha$-twisted sheaf $F=\{F_i,\psi_{ij}\}$
tensored with $L$ then naturally leads to the untwisted sheaf $\tilde F\coloneqq \pi^*F\otimes L$.
Analogously, any $\alpha^{-1}$-twisted sheaf $F$ can be turned into the untwisted sheaf
$\pi^*F\otimes L^*$.
The construction yields a functor $\Db(Z,\alpha)\to \Db(Y )$ which  in fact defines an equivalence of $\Db(Z,\alpha)$
with a full subcategory 
$$\tilde{(~)}\colon\Db(Z,\alpha)\congpf\Db(Y/Z)\subset\Db(Y).$$ 

The construction applied to $E$ itself  yields the sheaf $G\coloneqq \tilde E$ that corresponds to the unique non-trivial extension class in $\Ext^1_{Y}
(\kt_\pi,\ko_Y)$ and $\Db(Y/Z)\subset\Db(Y)$ can alternatively be described as the full subcategory of all objects $H$ for which the adjunction map
$\pi^*\pi_*(G^*\otimes H)\to G^*\otimes H$ is an isomorphism. Analogously, $\Db(Z,\alpha^{-1})$ is equivalent to the full subcategory of objects
$H$ for which $\pi^*\pi_*(G\otimes H)\congpf G\otimes H$.

We apply this construction  to the twisted K3 surface $(S,\alpha)$ and consider $Y=\PP(E)\to S$ as above. 
Assume $\alpha$ is of order $r$ and choose a lift $B=(1/r)B_0$ with $B_0\in H^2(S,\ZZ)$ of it.
The FM-kernel of our given equivalence $\Phi_P\colon\Db(S,\alpha)\congpf\ka_X\subset\Db(X)$, which is an object in $\Db((S,\alpha^{-1})\times X)$,
is turned into the untwisted sheaf $\tilde P\coloneqq\pi^*P\otimes(L^*\boxtimes\ko)$ on $Y\times X$.
This leads to the commutative diagram

$$\xymatrix{\Db(S,\alpha)\ar[d]^{\tilde{(~)}}\ar[r]^-{\pi_1^*}& \Db((S,\alpha)\times X)\ar[d]^{\tilde{(~)}}\ar[r]^{\otimes P}&\Db(S\times X)\ar[d]^-{\pi^*}\ar[r]^{\pi_{2*}}&\Db(X)\ar[d]^=\\
\Db(Y/S)\ar[r]^-{\pi_1^*}&                                          \Db((Y\times X)/(S\times X))\ar[r]^{\otimes \tilde P}&              \Db((Y\times X)/(S\times X))\ar[r]^-{\pi_{2*}}&\Db(X).
}$$
Therefore, the FM-functor $\Phi_P\colon\Db(S,\alpha)\congpf\ka_X\subset\Db(X)$ can be written as the composition $\Phi_P=
\Phi_{\tilde P}\circ\Phi_Q$ of a twisted FM-functor $\Phi_{Q}\coloneqq \tilde{(~)}$, with $Q= (\ko_S\boxtimes L)|_{\Gamma_\pi}$,
and an untwisted FM-functor $\Phi_{\tilde P}$:
\begin{equation}\label{eqn:compFM}
\xymatrix{\Phi_P\colon \Db(S,\alpha)\ar[r]^-{\Phi_Q}&\Db(Y)\ar[r]^-{\Phi_{\tilde P}}&\Db(X).}
\end{equation}
This allows one to decompose the diagram
(\ref{eqn:remcom}) as 
\begin{equation}\label{eqn:remcom2}
\xymatrix{\HoH_0(S,\alpha)\ar[d]_{I^B}^-\wr\ar[r]^{\Phi_Q^{\HoH_*}}&\HoH_0(Y)\ar[r]^{\Phi_{\tilde P}^{\HoH_*}}\ar[d]^-\wr&\HoH_0(X)\ar[d]^-\wr\\
\bigoplus H^{p,p}(S)\ar[r]^{\Phi_Q^{H,B}}&\bigoplus H^{p,p}(Y)\ar[r]^{\Phi_{\tilde P}^{H}}&\bigoplus H^{p,p}(X).}
\end{equation}
The right hand square is  induced by the usual untwisted FM-functor $\Phi_{\tilde P}$ and its commutativity
therefore follows from the result of Macr\`i and Stellari \cite[Thm.\ 1.2]{MS}. Hence, it suffices to prove the commutativity of the 
left hand square (which does not involve the cubic $X$ anymore). For greater clarity we split this further by decomposing $\Phi_Q$ as
$$\xymatrix{\Phi_Q\colon\Db(S,\alpha)\ar[r]^-{\pi^*}&\Db(Y,\pi^*\alpha)\ar[r]^-{L\otimes}&\Db(Y).}$$
Let us first consider $\pi^*\colon\Db(S,\alpha)\to\Db(Y,\pi^*\alpha)$ and the induced diagram
$$\xymatrix{\HoH_0(S,\alpha)\ar@/_3pc/[ddd]_{I^B}\ar[d]^-\wr\ar[r]&\HoH_0(Y,\pi^*\alpha)\ar@/^3pc/[ddd]^{I^{\pi^*B}}\ar[d]^-\wr\\
\HoH_0(S)\ar[d]_I\ar[r]&\HoH_0(Y)\ar[d]^I\\
H^*(S)\ar[d]^{\exp(B)}\ar[r]^{\sqrt{\rm td_\pi}\cdot\pi^*}&H^*(Y)\ar[d]_{\exp(\pi^*B)}\\
H^*(S)\ar[r]_{\sqrt{\rm td_\pi}\cdot\pi^*}&H^*(Y).}$$
Note that the usual $\sqrt{\rm td_\pi}\cdot\pi^*$ on the bottom is indeed the map on cohomology induced by  the functor $\pi^*\colon\Db(S,\alpha)\to\Db(Y,\pi^*\alpha)$,
which a priori depends on the choice of the lifts of $\alpha$ and $\pi^*\alpha$ to classes in $H^2(S,\QQ)$ and
$H^2(Y,\QQ)$, respectively, for which we choose
$B$ and $\pi^*B$. The commutativity of the upper and  the lower squares is trivial.
The commutativity of the middle square is an easy case of \cite[Thm.\ 1.2]{MS}.
Next consider $\Psi\coloneqq L\otimes(~)\colon\Db(Y,\alpha)\to\Db(Y)$ and the induced diagram
(where $\psi$ is defined by the requirement of commutativity)
$$\xymatrix{\HoH_0(Y,\pi^*\alpha)\ar @{} [dr] |{\circlearrowleft}\ar[d]^-\wr\ar[r]^-{\Psi^{\HoH_*}}&\HoH_0(Y)\ar[d]^-=\\
\HoH_0(Y)\ar[d]_I\ar[r]^\psi&\HoH_0(Y)\ar[dd]^I\\
H^*(Y)\ar[d]_{\exp(\pi^*B)}& \\
H^*(Y)\ar[r]^{\Psi^{H,\pi^*B}}&H^*(Y).}$$
By definition, $\Psi^{H,\pi^*B}$ is given by multiplication with ${\rm ch}^{\pi^*(-B)}(L)=\exp(-\pi^*B)\cdot\exp({\rm c}_1(L))$.
Here, use that $L^r$ is an untwisted line bundle and define ${\rm c}_1(L)\coloneqq(1/r){\rm c}_1(L^r)\in H^{1,1}(Y,\QQ)$. See
\cite[Sec.\ 1]{HuSt} for the conventions concerning twisted Chern classes. In particular, $\Psi^{H,\pi^*B}\circ\exp(\pi^*B)=\exp({\rm c}_1(L))$
and, therefore, it suffices to prove the commutativity of the diagram
\begin{equation}\label{eqn:untwHKR}
\xymatrix{\HoH_0(Y)\ar[d]_I\ar[rr]^\psi&&\HoH_0(Y)\ar[d]^I\\
H^*(Y)\ar[rr]^{\exp({\rm c}_1(L))}&& H^*(Y),}
\end{equation}
which no longer depends on $B$ and is a special case of Lemma \ref{lem:untwHKR} below.

This concludes the proof  of the commutativity of the diagram (\ref{eqn:remcom}) and hence of \cite[Prop.\ 6.2]{AT} in our
twisted setting. More precisely, if a first order deformation of $X$ in $D_d$ given by a class $\kappa_X\in H^1(T_X)$
corresponds via the interpretation of $D_d$ as period domain for $X$ and $S$
to a first order deformation $\kappa_S\in H^1(T_S)$, then $\Phi^{\HoH^2}\colon
\HoH^2(X)\to\HoH^2(S,\alpha)$ sends $\kappa_X$ to $\kappa_S$.

\medskip

%Most of the necessary modifications in \cite[Sec.\ 7]{AT} have all been spelled out by Reinecke in \cite{Reinecke}.
%In particular, Section 4.2 contains a discussion of Atiyah classes in the twisted case and the all important \cite[Cor.\ 7.5]{AT},
%which goes back to Toda, is \cite[Lem.\ 4.12]{Reinecke}. The deformation-obstruction theory for complexes of twisted sheaves
%is contained in \cite{Lieb}. As far as I am aware of, there is no account of the deformation-obstruction theory for complexes
%of twisted sheaves, in particular a description of the obstruction class as a product of Kodaira--Spencer class and Atiyah class in \cite{HT}.
%A closer inspection however reveals that rewriting all of it for twisted sheaves is possible.

To conclude the proof one has to prove that the kernel $P\in\Db((S,\alpha^{-1})\times X)$ deforms sideways, for which we
again apply Yoshioka's untwisting technique. Instead of attempting to deform the twisted $P$ sideways with $(S,\alpha)\times X$ we deform
the untwisted $\tilde P$. As the condition describing the full subcategory $\Db((S,\alpha^{-1})\times X)\cong\Db((Y\times X)/(S\times X))\subset \Db(Y\times X)$ is open, any deformation of $\tilde P$ will automatically induce a deformation of $P$.\footnote{This is confirmed by the observation that under the natural isomorphisms
$$\Ext^2_{(S,\alpha^{-1})\times X}(P,P)\cong\Ext^2_{(Y,\pi^*\alpha^{-1})\times X}(\pi^*P,\pi^*P)\cong\Ext^2_{Y\times X}(\tilde P,\tilde P)$$
the obstruction $o(P)\in \Ext^2_{(S,\alpha^{-1})\times X}(P,P)$  to deform $P$ sideways to first order is  first mapped
to $o(\pi^*P)$ and then to $o(\tilde P)-{\rm id}_{\pi^*P}\otimes o(\ko_\pi(-1))$. The latter,  however,
equals the obstruction $o(\tilde P)\in  \Ext^2_{Y\times X}(\tilde P,\tilde P)$ for $\tilde P$, because $\ko_\pi(-1)$ clearly deforms sideways.}
The decomposition (\ref{eqn:compFM}) leads to a diagram
$$\xymatrix{\HoH^2(X)=\Ext^2_{X\times X}(\ko_{\Delta_X},\ko_{\Delta_X})\ar[r]&\Ext^2_{Y\times X}(\tilde P,\tilde P)\ar[r]^-\sim&\Ext^2_{(S,\alpha^{-1})\times X}(P,P)\\
&\Ext^2_{Y\times Y}(\ko_{\Delta_Y},\ko_{\Delta_Y})\ar@{=}[d]\ar[u]\ar@{-->}[r]&\Ext^2_{(S,\alpha^{-1})\times (S,\alpha)}(\ko_{\Delta_S},\ko_{\Delta_S})\ar@{=}[d]\ar[u]_-\wr\\
&\HoH^2(Y)\ar@{-->}[r]&\HoH^2(S,\alpha).}$$
Recall that $\Phi_R^{\HoH_*}$ is defined for any FM-functor $\Phi_R$, whereas in order to define $\Phi_R^{\HoH^*}$ one
needs $\Phi_R$ to be fully faithful, which is the case for $\Phi_P$ and $\Phi_Q=\tilde{(~)}$. So, both maps
in 
\begin{eqnarray*}\HoH^2(X)\to &\HoH^2(S,\alpha)&\lto\HoH^2(Y)\\
\kappa_X\mapsto&\kappa_S&\mapslto\kappa_Y
\end{eqnarray*} are  well defined, where as above
$\kappa_X\in H^1(T_X)\subset\HoH^2(X)$ corresponds to $\kappa_S\in H^1(T_S)\subset
\HoH^2(S,\alpha)$ (via their periods or, equivalently, via $\Phi^{\HoH^2}$)
and $\kappa_Y$ is determined by our  pre-chosen deformation $E_t$ of $E$.

Now by \cite{HT} the obstruction $o(\tilde P)$ can be expressed as
$$o(\tilde P)=(\kappa_Y,\kappa_X)\circ{\rm At}(\tilde P).$$
(Unfortunately, an analogous formula in the twisted case  is not available.)
 The crucial \cite[Thm.\ 7.1]{AT}, which goes back to
Toda \cite{Toda}, proves that in the untwisted case $o(P)=0$ if $\kappa_X$ is mapped to $\kappa_S$ under
$\HoH^2(X)\to \HoH^2(S)$. However, in the twisted situation one has to face the additional problem  that there is no 
natural map $\HoH^2(X)\to\HoH^2(Y)$. Nevertheless, the argument in \cite{AT} goes through
essentially unchanged as follows. Using the same notation, one writes
$$o(\tilde P)=\pi_1^*\kappa_Y\circ {\rm At}_Y(\tilde P)+\pi_2^*\kappa_X\circ {\rm At}_X(\tilde P)\in\Ext^2(\tilde P,\tilde P).$$
The first term is the image  of $\pi_1^*\kappa_Y\circ{\rm At}_1(\ko_{\Delta_Y})\in \Ext^2(\ko_{\Delta_Y},\ko_{\Delta_Y})=\HoH^2(Y)$
which is just $\kappa_Y$, whereas the second one is the image of $-\pi_1^*\kappa_X\circ{\rm At}_2(\ko_{\Delta_X})\in\Ext^2(\ko_{\Delta_X},\ko_{\Delta_X})=\HoH^2(X)$ which is just $-\kappa_X$. 
Hence, to compare $\kappa_X$ and $\kappa_Y$ we do not need a map $\HoH^2(X)\to \HoH^2(Y)$
(which simply does not exist naturally), as we only need to compare their images in 
$\Ext^2(\tilde P,\tilde P)\cong\HoH^2(S,\alpha)$. Therefore it suffices to ensure that under 
$\HoH^2(X)\to\HoH^2(S,\alpha)$ the class $\kappa_X$ is mapped to $\kappa_S$, which was verified above.

\medskip

This concludes the argument proving that the FM-kernel $P$ deforms to first order with $(S,\alpha)\times X$.  The arguments 
in  \cite[Sec.\ 7.2]{AT} proving the existence of deformations of $P$ to all orders apply verbatim. Note that at the very end of the argument
one needs to apply a result of Lieblich \cite{Lieb2} saying that  the space of objects with no negative self-Exts in the derived category is an Artin stack of locally
finite presentation. Again, the result as such does not seem to be in the literature for the  twisted situation, but once again it can be 
deduced from the untwisted case by Yoshioka's trick.
\end{proof}

It remains to check the commutativity of (\ref{eqn:untwHKR}) which is a general fact. Consider a smooth variety $Z$ and
$\alpha_{ijk}\coloneqq\beta_{ij}\cdot\beta_{jk}\cdot\beta_{ki}$ with $\beta_{ij}\in\ko_{U_{ij}}^\ast$. The associated Brauer class
$\alpha\in H^2(Z,\ko_Z^\ast)$ is of course trivial and hence $\Db(Z,\{\alpha_{ijk}\})$ and $\Db(Z)$ are equivalent categories
and an explicit equivalence can be given by `untwisting by $\{\beta_{ij}\}$', i.e.\ by $E=\{E_i,\varphi_{ij}\}\mapsto \{E_i,\varphi_{ij}\cdot\beta_{ij}^{-1}\}$.
Note that  changing $\beta_{ij}$ by a cocycle $\{\delta_{ij}\}$, which would correspond to an untwisted line bundle say $M$, the
equivalence would be modified by $M^*\otimes(~)$.

Assume furthermore that $\alpha_{ijk}^r=1$. Then $\{\beta_{ij}^r\}$ is a cocycle defining a line bundle $H$ and
we define ${\rm c}_1(\beta)\coloneqq (1/r){\rm c}_1(H)\in H^{1,1}(Z)$. Explicitly, ${\rm c}_1(\beta)=\{d\log\beta_{ij}\}$.
 
\begin{lem}\label{lem:untwHKR}
The `untwisting by $\{\beta_{ij}\}$', i.e.\ the equivalence
$$\Phi\colon\Db(Z,\{\alpha_{ijk}\})\congpf\Db(Z),~E=\{E_i,\varphi_{ij}\}\mapsto \{E_i,\varphi_{ij}\cdot\beta_{ij}^{-1}\},$$
induces a commutative diagram
$$\xymatrix{\HoH_*(Z,\{\alpha_{ijk}\})\ar[r]^-{\Phi^{\HoH_*}}\ar[d]_{\rm HKR}&\HoH_*(Z)\ar[d]^{\rm HKR}\\
H^*(Z)\ar[r]_{\exp({\rm c}_1(\beta))}&H^*(Z)}$$
\end{lem}

The commutativity of (\ref{eqn:untwHKR}) then follows from the observation that $L\otimes(~)$ can be written as
the composition of the `untwisting by $\{\beta_{ij}\}$' as above with the equivalence $ \kl\otimes(~)$. Here, $\kl$ is the untwisted line bundle given by
$\{\psi_{ij}\cdot\beta_{ij}\}$, where $L$ itself is the $\{\alpha_{ijk}^{-1}\}$-twisted line bundle given by $\{\psi_{ij}\}$.

Indeed, for $\Psi\coloneqq \kl\otimes(~)\colon\Db(Z)\congpf\Db(Z)$  the commutativity of
$$\xymatrix{\HoH_*(Z)\ar[rr]^-{\Psi^{\HoH_*}}\ar[d]_{\rm HKR}&&\HoH_*(Z)\ar[d]^{\rm HKR}\\
H^*(Z)\ar[rr]_{\exp({\rm c}_1(\kl))}&&H^*(Z)}$$
is an easy special case of \cite[Thm.\ 1.2]{MS}\footnote{Note that ${\rm td}^{1/2}\wedge$ can be dropped here and in the lemma, as it commutes with $\exp({\rm c}_1(\kl))$.}, which can be proved by a direct calculation. The proof of the lemma is a variant of this computation.

\begin{proof}
Consider the universal Atiyah class ${\rm At}\colon\ko_\Delta\to\Delta_*\Omega_Z[1]$. Twisted with a line bundle of the form
$M\boxtimes M^*$ it yields a map ${\rm At}_M\colon\ko_\Delta\to\Delta_*\Omega_Z[1]$. The usual formula ${\rm c}_1(E\otimes M)={\rm c}_1(E)+\rk\, E\cdot{\rm c}_1(M)$ corresponds to the universal formula ${\rm At}_M=\alpha+\Delta_*{\rm c}_1(M)$, which can be checked
by using arguments of \cite{BF1,BF2} or a direct cocycle computation. Here, ${\rm c}_1(M)$
is viewed as a map $\ko_Z\to \Omega_Z[1]$ which can be pushed forward via $\Delta$. Similarly, the exponential
$\exp({\rm At})\colon\ko_\Delta\to\bigoplus \Delta_*\Omega_Z^i[i]$ (see \cite{Cald}) twisted with $M\boxtimes M^*$
is given by $\exp({\rm At})_M=\Delta_*\exp({\rm c}_1(M))\circ\exp({\rm At})$.

Let now $f\in\HoH_j(Z)=\Ext^j_{Z\times Z}(\Delta_!\ko_Z,\Delta_*\ko_Z)$ and denote by $F\in\Ext^j_Z(\ko_Z,\Delta^*\Delta_*\ko_Z)$
 its image under $\Ext^j_{Z\times Z}(\Delta_!\ko_Z,\Delta_*\ko_Z)\cong\Ext^j_Z(\ko_Z,\Delta^*\Delta_*\ko_Z)$.
 So if  $\eta\colon (\Delta_!\Delta^*)\Delta_*\ko_Z\to\Delta_*\ko_Z$ denotes  adjunction, then $f=\eta\circ\Delta_!F$.
Due to \cite[Prop.\ 4.4]{Cald}, the latter is under the HKR isomorphism given by $\exp({\rm At})$,
so $$\xymatrix{\eta\colon (\Delta_!\Delta^*)\Delta_*\ko_Z\cong\bigoplus\Delta_*(\Omega^i[i]\otimes\omega_Z^{-1}[-d]))\cong
\bigoplus\Delta_* (\Omega_Z^{d-i})^*[i-d]\ar[rr]^-{\exp({\rm At})}&&\Delta_*\ko_Z.}$$ The image of $f$ under $\kl\otimes(~)$ is given by
tensoring with $\kl\boxtimes \kl^*$. The push-forward $\Delta_!F$ remains unchanged by tensoring with $\kl\boxtimes\kl^*$
and by the above $\eta$ changes by composing with $\Delta_*\exp({\rm c}_1(\kl))$.

Literally the same argument applies to the untwisting by $\{\beta_{ij}\}$ for which one has to observe that 
the universal Atiyah class ${\rm At}\colon\ko_\Delta\to\Delta_*\Omega_Z[1]$ on $(Z,\{\alpha_{ijk}^{-1}\})\times (Z,\{\alpha_{ijk}\})$
under untwisting by $\{\beta_{ij}\}$ becomes ${\rm At}+\Delta_*{\rm c}_1(\beta)\colon\ko_\Delta\to\Delta_*\Omega_Z[1]$ on $Z\times Z$.
\end{proof}

%%%%%%%%%%%%%%%%%%%%%%%%%%%%%%%
\section{Proofs}\label{sec:Proofs}

%%%%%%%%%%%%%%%%%%%%%%%%%%%%%
\subsection{Proof of Theorem \ref{thm:noFMvg}}\label{sec:ProofAut}

i) According to Corollary \ref{cor:Z3Z}, for every smooth cubic $X\subset\PP^5$ there exists a distinguished FM-autoequivalence
$\Phi_0\colon\ka_X\congpf\ka_X$ of infinite order which acts as the identity on $T(\ka_X)$, so it is symplectic, and such that  $\Phi_0^3$
is the double shift $E\mapsto E[2]$. We have to show that for the  very general cubic every symplectic FM-equivalence $\Phi$
is a power of $\Phi_0$.

As $\widetilde H^{1,1}(\ka_X,\ZZ)\cong A_2$ for  very general $X$  and $\Phi^H={\rm id}$ on $T(\ka_X)=A_2^\perp$,
the induced action $\Phi^H$ is contained in ${\rm O}(A_2)$. Clearly, any Hodge isometry of $\widetilde H(\ka_X,\ZZ)$
that is the identity on $A_2^\perp$ stays a Hodge isometry for all deformations of $X$. Therefore, applying
Theorem \ref{thm:Defo1}, $\Phi$ deforms to FM-autoequivalences $\Phi_t\colon\ka_{\kx_t}\cong \ka_{\kx_t}$ for
cubics $\kx_t$ in a Zariski open neighbourhood
$U\subset\kc$ of $X$ inside the moduli space of smooth cubics.

Then for all but finitely many $d$ satisfying ($\ast$$\ast$) the intersection $U\cap\kc_d$ is non-empty (and open) and, therefore,
by \cite[Thm.\ 1.1]{AT} there exists $t\in U$ such that $\ka_{\kx_t}\cong\Db(S)$ for some K3 surface $S$. Due to \cite[Thm.\ 2]{HMS}, autoequivalences
of $\Db(S)$ are orientation preserving and hence $\Phi^H\in{\mathfrak A}_3\cong\ZZ/3\ZZ$, cf.\
Remark \ref{rem:OA_2}.  Thus, by composing with some power of $\Phi_0$, we may assume that $\Phi^H={\rm id}$.

Now apply Corollary \ref{cor:locustw} and Theorem \ref{thm:genastastast}, to be proved below,
to conclude that there exists $t\in U$ such that $\ka_{\kx_t}\cong\Db(S,\alpha)$ for a twisted K3 surface $(S,\alpha)$ not admitting any $(-2)$-class.
Indeed, $$(\kc_{\rm K3'}\cap U)\setminus\kc_{\rm sph}\ne\emptyset,$$
where $\kc_{\rm K3'}\coloneqq\bigcup_{{\rm (\ast\ast')}}\kc_d\subset\kc$ and $\kc_{\rm sph}\subset\kc$ is the image of $D_{\rm sph}$.
By \cite[Thm.\ 2]{HMS2}, we know that then $\Phi_t$ is isomorphic to an even shift $E\mapsto E[2k]$. It is easy to see that $k$
is independent of $t$.

The locus of points $U_0\subset U$ such that $\Phi_t\cong[2k]$ for $t\in U_0$ is Zariski open and by the above non-empty.
Therefore, for every $X\in \kc$ in the intersection of all $U_0\subset\kc$ the assertion holds. 
But this intersection is certainly countable, as FM-kernels  are parametrized by countably many products of Quot-schemes.

\medskip

ii) Now consider a non-special cubic $X$, i.e.\ $X\in\kc\setminus\bigcup \kc_d$, and an arbitrary $\Phi\in\Aut(\ka_X)$. By composing with the shift functor
$[1]$ if necessary, we may assume that $\Phi^H$ acts
trivially on the discriminant group $A_{A_2}\cong A_{A_2^\perp}$. But then the induced Hodge isometry
of $T(\ka_X)\cong A_2^\perp$ extends to a Hodge isometry of $H^4(X,\ZZ)$ that respects $h$. By the Global Torelli
theorem \cite{Voisin,Loo,Char2} it is  therefore induced by an  automorphism $f\in\Aut(X)$, which clearly acts trivially on the orthogonal
complement of $h^\perp\subset H^*(X,\ZZ)$ and hence as the identity on $A_2\subset\widetilde H(\ka_X,\ZZ)$. Moreover, since $f$ respects
$H^{3,1}(X)$,  the action of $f$ in $\widetilde H(\ka_X,\ZZ)$ preserves the orientation.

So, composing, if necessary, $\Phi$ with the shift functor and an automorphism,
we reduce to the case $\Phi\in\Aut_s(\ka_X)$. As $X$ is non-special, i.e.\ $A_2\cong\widetilde H^{1,1}(\ka_X,\ZZ)$, we
can deform $\Phi$ sideways as above until it can be interpreted as an autoequivalence of a category of the form $\Db(S)$,
which implies that it is orientation preserving.
This eventually proves that for every non-special cubic the image
of $\rho\colon\Aut(\ka_X)\to\Aut(\widetilde H(\ka_X,\ZZ))$ is the subgroup of orientation preserving Hodge isometries.
\qed

\begin{remark}
We expect the first assertion in Theorem \ref{thm:noFMvg} to hold for every non-special cubic, i.e.\ for all $X\in\kc\setminus\bigcup\kc_d$, but  this would require to show that  if $\Phi\in\Aut(\ka_X)$ deforms to the identity functor and $\widetilde H^{1,1}(\ka_X ,\ZZ)\cong A_2$,
then $\Phi\cong{\rm id}$. The techniques of \cite{HMS2} should be useful  here, but they require the existence of stability conditions.

Furthermore, one would also expect that any $\Phi\in\Aut(\ka_X)$ of any cubic preserves the natural orientation.
%\footnote{The
%referee suggested to argue again by deformation: Consider deformations $\kx_t$, $t\in T$
%that keep the whole $\widetilde H^{1,1}(\ka_X,\ZZ)$ algebraic. This deformation space, when positive, should intersect $\bigcup%\kc_d$ for $d$ satisfying ($\ast$$\ast$)
%in a dense set. On the other hand, using Theorem \ref{}, any FM-equivalence $\Phi$ deforms sideways in this family
%$\k
\end{remark}

%%%%%%%%%%%%%%%%%%%
\subsection{Proof of Theorem \ref{thm:genastastast}}\label{sec:thm:genastastast}
Assertion i) follows from Theorem \ref{thm:twHas} and Proposition \ref{prop:indHodge}. For the converse,
fix $d$ satisfying ($\ast$$\ast'$). Then for any smooth cubic $X\in\kc_d$ there exists
a Hodge isometry 
\begin{equation}\label{eqn:HIstart}
\varphi\colon\widetilde H(S,\alpha,\ZZ)\congpf\widetilde H(\ka_X,\ZZ)
\end{equation}
for some twisted K3 surface $(S,\alpha)$. In fact, this Hodge isometry can be chosen globally
over the period domain $D_d$ (or some appropriately constructed covering $\tilde\kc_d$ of $\kc_d$, see \cite{AT}).
The aim is to show that generically this Hodge isometry is induced by an equivalence
$\ka_X\cong\Db(S,\alpha)$ (up to changing the orientation).

The starting point for the argument is \cite[Thm.\ 4.1]{AT}, which is based on Kuznetsov's work \cite{Kuz1} and on
the description of the image
of the period map for cubic fourfolds due to Laza \cite{Laza} and Looijenga \cite{Loo}. Combined, these results show that
for every $d$ satisfying ($\ast$$\ast'$) (but in fact  ($\ast$) is enough) there exists a smooth cubic
$X\in\kc_8\cap\kc_d$, a K3 surface $S_0$ and an equivalence
$$\Phi_0\colon\ka_X\congpf\Db(S_0).$$
By \cite{AT} or Proposition \ref{prop:indHodge}, any such $\Phi_0$ induces
a Hodge isometry $\Phi_0^H\colon\widetilde H(\ka_X,\ZZ)\congpf\widetilde H(S_0,\ZZ)$ (usually completely unrelated to (\ref{eqn:HIstart})). Consider now the composition
$$\psi\coloneqq\Phi_0^H\circ\varphi\colon\widetilde H(S,\alpha,\ZZ)\congpf \widetilde H(\ka_X,\ZZ)\congpf \widetilde H(S_0,\ZZ).$$
By modifying $\varphi$ (globally over $D_d$) if necessary (use Lemma \ref{lem:revor}.), we may assume that  $\psi$ preserves the orientation and
then \cite{HuSt2} applies and shows that there exists an equivalence
$\Psi\colon\Db(S,\alpha)\congpf\Db(S_0)$ with $\Psi^H=\psi$.
Then the equivalence $$\Phi\coloneqq\Phi_0^{-1}\circ\Psi\colon\Db(S,\alpha)\congpf\Db(S_0)\congpf\ka_X$$
satisfies $\Phi^H=\varphi$.

We can now forget about $S_0$ and only keep $X$ and $S$ and the equivalence $\Phi=\Phi_P$ with
$P\in\Db((S,\alpha^{-1})\times X)$. Then consider the two families $\kx_t$ and $(\ks_t,\alpha_t)$ over
$D_d$ (or rather $\tilde\kc_d$ in order to use the Zariski topology) of cubics resp.\ twisted K3 surfaces with $X=\kx_0$, $S=\ks_0$,
for which $\varphi$ defines Hodge isometries $\widetilde H(\ks_t,\alpha_t,\ZZ)\congpf \widetilde H(\ka_{\kx_t},\ZZ)$ for all $t$.
As $\Phi\colon\Db(S,\alpha)\congpf\ka_X$ induces $\varphi$, Theorem \ref{thm:Defo2} applies and shows that $\Phi$ can be deformed
to equivalences $\Phi_t\colon\Db(\ks_t,\alpha_t)\congpf\ka_{\kx_t}$ for all $t$ in a Zariski open neighbourhood of $0\in\tilde\kc_d$.\qed

%Furthermore, under 
%\begin{eqnarray*}\Ext^1_{(S,\alpha^{-1})\times X}(P,P\otimes\Omega_{S\times X})&\cong&\Ext^1_{(Y,\pi^{-1}\alpha)}(\pi^*P,\pi^*P\otimes
%\pi^*\Omega_{S\times X})\\
%&\to&\Ext^1_{(Y,\pi^{-1}\alpha)}(\pi^*P,\pi^*P\otimes \Omega_{Y})\\
%&\to&\Ext^1_{Y}(\tilde P,\tilde P\otimes \Omega_{Y})
%\end{eqnarray*} the Atiyah class ${\rm At}(P)$ is first mapped to ${\rm At}(\pi^*P)$ and then to ${\rm At}(\tilde P)-{\rm id}_{\pi^*P}%\otimes{\rm At}(\ko_\pi(-1))$.
%The pull-back of a Kodaira--Spencer class $\kappa_{S\times X}\in H^1(T_{S\times X})$ is under
%$H^1(T_Y)\to H^1(T_{S\times X})$ the image of the Kodaira--Spencer class $\kappa_Y$ determined by the given first order deformation of $E$ sideways.
%
%Therefore, under the natural isomorphism $\Ext^2_{(S,\alpha^{-1})\times X}(P,P)\cong\Ext^2_Y(\tilde P,\tilde P)$ the class
%$\kappa_{S\times X}\circ {\rm At}(P)$ is mapped to $\kappa_Y\circ{\rm At}(\pi^*P)$ which
%by the above coincides with $\kappa_Y\circ{\rm At}(\tilde P)-\kappa_Y\circ({\rm id}_{\pi^*P}\otimes{\rm At}(\ko_\pi(-1)))$.
%As $\kappa_Y\circ{\rm At}(\ko_\pi(-1))$ is trivial\footnote{??TBC}, the second summand vanishes,
%and hence $\kappa_{S\times X}\circ {\rm At}(P)$ is mapped to $\kappa_Y\circ{\rm At}(\tilde P)=o(\tilde P)$ which is also the image of %$o(P)$. Therefore, $$\kappa_{S\times X}\circ {\rm At}(P)=o(P),$$ which generalizes
%\cite{HT} to the twisted case. Hence, all arguments of \cite[Sec.\ 7]{AT} (combined with \cite{Reinecke}) go through.

%%%%%%%%%%%%%%%%%%%

\subsection{Proof of Theorem \ref{thm:HodgeA}}\label{sec:proofthm:HodgeA} 
The first assertion of the theorem has been proved already
as  Corollary \ref{cor:verygeneralnoFM}.  For ii) and iii)  recall that  any FM-equivalence $\ka_X\cong\ka_{X'}$  induces a Hodge isometry 
$\widetilde H(\ka_X,\ZZ)\cong\widetilde H(\ka_{X'},\ZZ)$, cf.\ Proposition \ref{prop:equivHodgeA}.
So it remains to prove the converse for generic $X\in\kc_d$ with $d$ satisfying ($\ast$$\ast'$)
resp.\  very general $X\in\kc_d$ for arbitrary $d$.
The first case is easy, as then, by Theorem \ref{thm:genastastast}, $\ka_X\cong\Db(S,\alpha)$ and $\ka_{X'}\cong\Db(S',\alpha')$ for twisted
K3 surfaces $(S,\alpha)$ resp.\ $(S',\alpha')$.
The assertion then follows from \cite{HuSt2} and Lemma \ref{lem:revor}.
%Note that in principle one could try to drop the assumption on $d$ in Theorem \ref{thm:HodgeA}
%by first deforming to special $X,X'\in\kc_d$ with $\ka_X\cong\Db(S,\alpha)$ and $\ka_X\cong\Db(S',\alpha')$.
%However, in order to deform  back to generic $X,X'$ one would need to set up the deformation theory 
%of \cite{AT}, cf.\ Section \ref{sec:defo}, for FM-equivalences $\ka_X\cong\ka_{X'}$, which would require a
%comparison of Hochschild homology $\HoH_*(\ka_X)$ and $\widetilde H(\ka_X,\CC)$ that has not been worked out.
%which does not require any twisted K3 surfaces and relies entirely on techniques from \cite{AT}. 
%More precisely, we shall show that there is a dense Zariski open subset $U\subset \kc_d$ such that
%for all $X\in U\setminus \bigcup_{d'\ne d}\kc_{d'}$ the existence of a Hodge isometry
%$\widetilde H(\ka_X,\ZZ)\cong\widetilde H(\ka_{X'},\ZZ)$ implies that 
%$\ka_X$ and $\ka_{X'}$ are FM-equivalent.

%Composing, if necessary, with an orientation reversing Hodge isometry as in Lemma \ref{lem:revor} yields
%a Hodge isometry that preserves the orientation of the four positive directions.
For the second case consider the correspondence 
$$Z\coloneqq\{(X,X',\varphi)~|~X\in\kc_d\text{ and }\varphi\colon \widetilde H(\ka_X,\ZZ)\congpf\widetilde H(\ka_{X'},\ZZ)\}$$
of smooth cubics $X,X'$ with $X\in\kc_d$  and a Hodge isometry $\varphi$. (Note that with $X$ also $X'\in\kc_d$.)
This correspondence consists of countably many components $Z_i\subset Z$
and for the image of a component $Z_0\subset Z$
under the first projection $\pi\colon Z\to \kc_d$ one either has   $\pi(Z_0)\subset \kc_d\cap\bigcup_{d'\ne d}\kc_{d'}$ or $\pi(Z_0)\subset\kc_d$
is dense.
% (a priori it could happen that only one of the two cubics becomes singular).

As we are interested in very general $X\in\kc_d$ only, we may assume that we are in the latter situation. Then  by \cite[Thm.\ 1.1]{AT}, cf.\ Section \ref{sec:ProofAut}, one finds a $(X,X',\varphi)\in Z_0$ for which there exist
K3 surfaces $S$ and $S'$ and FM-equivalences
\begin{equation}\label{eqn:DBA2}\Psi\colon\ka_X\congpf\Db(S)\text{ and }\Psi'\colon\ka_{X'}\congpf\Db(S').
\end{equation}
By Proposition \ref{prop:indHodge}, $\Psi$ and $\Psi'$ induce Hodge isometries $\Psi^H$ resp.\
$\Psi'^H$, which composed with  $\varphi$ yield a   Hodge isometry
$$\xymatrix{\varphi_0\colon\widetilde H(S,\ZZ)\ar[r]^-\sim_{(\Psi^{-1})^H}&\widetilde H(\ka_X,\ZZ)\ar[r]^-\sim_-{\varphi}&\widetilde H(\ka_{X'},\ZZ)\ar[r]^-\sim_{\Psi'^H}&\widetilde H(S',\ZZ).}$$
We may assume that $\varphi_0$ is orientation preserving and, thus, induced by a FM-equivalence  $\Phi_0\colon
\Db(S)\congpf\Db(S')$. Composing the latter with the equivalences
(\ref{eqn:DBA2}) yields a FM-equi\-valence $\Phi\colon\ka_X\cong\ka_{X'}$  inducing $\varphi$.
Now use Theorem \ref{thm:Defo1} to deform $\Phi$ sideways to FM-equivalences
$\Phi_t\colon\ka_{X_t}\congpf\ka_{X'_t}$ for all points $(X_t,X'_t,\varphi_t\equiv\varphi)$ in a 
Zariski dense open subset $U_0\subset Z_0$. 

Hence, for all $X\in \bigcap \pi(U_i)$, with the intersection over all components $Z_i\subset Z$ (dominating $\kc_d$),
the existence of a Hodge isometry $\widetilde H(\ka_X,\ZZ)\cong\widetilde H(\ka_{X'},\ZZ)$ implies the existence of 
a FM-equivalence $\ka_X\cong\ka_{X'}$.
\qed

%%%%%%%%%%%%%%%%%%%%%%

\end{document}